\documentclass[reqno, 11pt]{amsart} 

\usepackage[utf8]{inputenc}
\usepackage[T1]{fontenc}

\usepackage[english]{babel}

\usepackage[margin=1.2in]{geometry}
\usepackage{varwidth}
\usepackage{enumitem}
\usepackage{xcolor}

\usepackage{amsmath, amsthm, amssymb, bbm}
\usepackage{unicode-math}
\usepackage[all]{xy}

\usepackage[bookmarks = true]{hyperref}

\usepackage{tikz}
\usetikzlibrary{cd}
\usetikzlibrary{trees}

\newtheorem{thm}{Theorem}[section]
\newtheorem*{thm*}{Theorem}
\newtheorem{lem}[thm]{Lemma}
\newtheorem{prop}[thm]{Proposition}

\theoremstyle{definition}
\newtheorem{rmk}[thm]{Remark}

\newtheorem{defn}[thm]{Definition}
\newtheorem{ex}[thm]{Example}

\numberwithin{equation}{section}

\newcommand{\overbar}[1]{\mkern 1.5mu\overline{\mkern-1.5mu#1\mkern-1.5mu}\mkern 1.5mu}

\newcommand{\ov}{\overline}

\newcommand{\mr}{\mathrm}

\newcommand{\wt}{\widetilde}

\newcommand{\tr}{\mathrm{tr}}

\newcommand{\Orb}{\operatorname{Orb}}

\newcommand{\diag}{\operatorname{diag}}

\newcommand{\Hom}{\operatorname{Hom}}

\newcommand{\Lie}{\operatorname{Lie}}

\newcommand{\Spec}{\operatorname{Spec}}

\newcommand{\GL}{\operatorname{GL}}

\newcommand{\tensor}{\otimes}

\newcommand{\iso}{\cong}

\newcommand{\lr}{\longrightarrow}

\newcommand{\mbA}{\mathbb{A}}

\newcommand{\mbC}{\mathbb{C}}

\newcommand{\mbN}{\mathbb{N}}

\newcommand{\mbQ}{\mathbb{Q}}
\newcommand{\mbR}{\mathbb{R}}

\newcommand{\mbU}{\mathbb{U}}

\newcommand{\mbZ}{\mathbb{Z}}

\newcommand{\mcA}{\mathcal{A}}

\newcommand{\mcE}{\mathcal{E}}

\newcommand{\mcS}{\mathcal{S}}

\newcommand{\mfP}{\mathfrak{P}}

\newcommand{\mfR}{\mathfrak{R}}
\newcommand{\mfS}{\mathfrak{S}}

\newcommand{\mfb}{\mathfrak{b}}

\newcommand{\mfo}{\mathfrak{o}}

\newcommand{\mfq}{\mathfrak{q}}

\newcommand{\mfs}{\mathfrak{s}}

\newcommand{\mfu}{\mathfrak{u}}

\newcommand{\back}{\backslash}

\newcommand{\Sym}{\mr{Sym}}
\newcommand{\Skew}{\mr{Skew}}

\newcommand{\rs}{\mr{rs}}

\newcommand{\KM}{\mr{KM}}
\newcommand{\Inv}{\mr{Inv}}

\newcommand{\simto}{\overset{\sim}{\to}}
\newcommand{\simlr}{\overset{\sim}{\lr}}

\title{Gaussian test functions and Jacquet--Rallis transfer}

\date{\today}

\author{Andreas Mihatsch}
\address{School of Mathematical Sciences, Zhejiang University, 866 Yuhangtang Rd, Hangzhou, 310058, P. R. China.}
\email{mihatsch@zju.edu.cn}

\author{Siddarth Sankaran}
\address{Department of Mathematics, University of Manitoba, Winnipeg, Manitoba, Canada.}
\email{siddarth.sankaran@umanitoba.ca}

\author{Tonghai Yang}
\address{Department of Mathematics, University of Wisconsin, Madison, WI 53706, USA.}
\email{thyang@math.wisc.edu}

\date{\today}

\begin{document}

\begin{abstract}
We construct Gaussian test functions for the general linear side of the Jacquet--Rallis relative trace formula comparison. These are functions which are defined in terms of their orbital integrals and transfer to the compact unitary group. Our construction relies on the formalism of Kudla--Millson and simple geometric properties of symmetric spaces. In particular, it also provides an explicit formula in terms of the Howe operator.
\end{abstract}

\maketitle
\tableofcontents

\section{Introduction}

The relative trace formula comparison of Jacquet--Rallis \cite{JR}, which lead to a proof of the unitary Gan--Gross--Prasad Conjecture \cite{BLZZ, BCZ}, relies on a comparison of orbital integrals between a $\GL_n$-setting and a unitary setting. The main local problems in this context are the existence of smooth transfer and the fundamental lemma.

Over non-archimedean fields, the existence of transfer was proved by Wei Zhang \cite{Zhang_GGP}. The fundamental lemma has three independent proofs: one via equal characteristic methods by Zhiwei Yun and Gordon \cite{Gordon, Yun}, one via global theta series by Wei Zhang \cite{Zhang_AFL}, and a local one based on Fourier transforms by Beuzart--Plessis \cite{BP}.

Over $\mbR$, it was proved by Hang Xue \cite{Xue} that a dense subspace of Schwartz functions is transferable which was sufficient for the global applications in \cite{BLZZ, BCZ}. It is conjectured that transfer exists for all Schwartz functions.

\subsection{Gaussian test functions}
In arithmetic situations, one is often interested in specific test functions at the archimedean place. For example, when studying the cohomology of Shimura varieties, one considers so-called Lefschetz functions, cf. \cite[\S3]{Kottwitz_lambda}. Their orbital integrals are non-zero only for elliptic group elements which means they can be understood as coming by transfer from the compact inner form of the group in question.

In the trace formula setting of Jacquet--Rallis, the analogous kind of functions are those coming by transfer from the compact unitary group $U(n+1)$. For example, a transfer of the identity function is used in Wei Zhang's relative trace formula approach to the Arithmetic Gan--Gross--Prasad Conjecture \cite{Zhang_12}. Such transfers were named \emph{Gaussian test functions} in \cite{RSZ_diagonal}, and these are the test functions from the title of the paper.

The purpose of our paper is to give a simple, direct, and local construction of Gaussian test functions for both Lie groups and Lie algebras. Their existence for the group was already known before by \cite[Proposition 4.11]{BLZZ}. Our explicit construction has the advantage that it also allows to study derivatives of orbital integrals. This matters, for example, during the proof of the arithmetic fundamental lemma \cite[\S12]{Zhang_AFL}, \cite[\S10]{MZ}; and also plays a key role in ongoing work of the authors on arithmetic generating series.

\subsection{Main result}
Our main result concerns the Lie algebra variant of the Jacquet--Rallis setting. Consider the tangent space at the identity to $\GL_{n+1}(\mbC)/\GL_{n+1}(\mbR)$:
$$\mfs_{n+1} = \{y \in M_{n+1}(\mbC) \mid y + \ov{y} = 0\}.$$
The group $\GL_n(\mbR)$ acts on it by conjugation via $g\mapsto \mr{diag}(g,1)$. We denote by $[\mfs_{n+1}]_\rs$ the set of regular semi-simple orbits.

In a similar way, for every hermitian $\mbC$-vector space $V$, the group $U(V)$ acts by conjugation on the Lie algebra $\mfu(V\oplus \mbC)$ of $U(V\oplus \mbC)$. Here, $\mbC$ is viewed with the standard hermitian form of signature $(1,0)$. We again denote by $[\mfu(V\oplus \mbC)]_\rs$ the set of regular semi-simple orbits. The orbit matching of Jacquet and Rallis \cite{JR} defines a bijection
\begin{equation}\label{eq:intro_matching_Lie}
[\mfs_{n+1}]_\rs\ \simlr\ \coprod_{r + s = n} [\mfu(V_{(r,s)}\oplus \mbC)]_\rs,
\end{equation}
where $V_{(r,s)}$ is a choice of hermitian $\mbC$-vector space of signature $(r,s)$.

Suppose $ϕ\in \mcS(\mfs_{n+1})$ is a Schwartz function and $y \in \mfs_{n+1}$ regular semi-simple. Jacquet and Rallis introduced the orbital integral
\begin{equation}\label{eq:intro_def_orb_int_Lie}
\Orb(y, ϕ) = ε(y) \int_{\GL_n(\mbR)} ϕ(g^{-1}yg) η(g)\,dg,
\end{equation}
where $η = \mr{sign}\circ \det$ is the sign character and $ε(y) \in \{\pm 1\}$ is a \emph{transfer factor} as in \S\ref{ss:orb_int_Lie}.

Both $\mfs_{n+1}$ and each $\mfu(V_{(r,s)}\oplus \mbC)$ are quadratic spaces with quadratic form $Q(z) = -\tr(z^2)$. If two elements $y\in \mfs_{n+1,\rs}$ and $x\in \mfu(V_{(r,s)}\oplus \mbC)_\rs$ match under \eqref{eq:intro_matching_Lie}, then $Q(x) = Q(y)$. Since $\mfu(V_{(n,0)}\oplus \mbC)$ is positive definite, we have the usual Gaussian
\begin{equation}\label{eq:intro_Gaussian_unit}
Ψ(x) = e^{-2πQ(x)},\quad\quad x\in \mfu(V_{(n,0)}\oplus \mbC).
\end{equation}

\begin{defn}\label{def:intro_Gaussian_Lie}
A \emph{Gaussian test function} on $\mfs_{n+1}$ is a Schwartz function $Φ\in \mcS(\mfs_{n+1})$ that is a smooth transfer of $Ψ$. That is, for all regular semi-simple $y\in \mfs_{n+1}$,
$$\Orb(y, Φ) = \begin{cases}
e^{-2πQ(y)} & \text{if $y$ matches to signature $(n,0)$ under \eqref{eq:intro_matching_Lie}}\\[1mm]
0 & \text{otherwise.}
\end{cases}$$
\end{defn}

Our main theorem constructs such Gaussian test functions from the differential forms of Kudla and Millson \cite{KM_harmonic_I, KM}. Namely, the action of $\GL_n(\mbR)$ on $\mfs_{n+1}$ is an orthogonal representation
\begin{equation}\label{eq:intro_rho}
\rho\colon \GL_n(\mbR)\lr SO(\mfs_{n+1}).
\end{equation}
The pullback $ρ^*(φ_\KM)$ of the Kudla--Millson form $φ_\KM$ naturally lies in $\mcS(\mfs_{n+1})\tensor \det(\Sym_n(\mbR))$. Let $ω\in \det(\Sym_n(\mbR))$ be the properly oriented generator that defines the Haar measure on $\GL_n(\mbR)$.
\begin{thm}[see Theorem \ref{thm:main body}]\label{thm:intro_main_Lie}
Let $Φ\in \mcS(\mfs_{n+1})$ be the Schwartz function characterized by the identity
\begin{equation}\label{eq:intro_key}
2^{n(n+1)/4}\cdot ρ^*(φ_\KM) = Φ \tensor ω.
\end{equation}
Then $Φ$ is a Gaussian test function in the sense of Definition \ref{def:intro_Gaussian_Lie}.
\end{thm}
In particular, Theorem \ref{thm:intro_main_Lie} exhibits a surprising connection between the Kudla program and the Gan-Gross-Prasad conjecture.

The Kudla--Millson form $φ_\KM$ has an explicit description in terms of the Howe operator. So, by extension, the function $Φ$ in \eqref{eq:intro_key} is explicit as well. We refer to Examples \ref{ex:1} and \ref{ex:2} for the cases $n = 1$ and $n = 2$.

\subsection{Transfer for polynomial type Schwartz functions}

A simple argument allows to extend Theorem \ref{thm:intro_main_Lie} to all \emph{polynomial type} Schwartz functions $ψ\in \mcS(\mfu(V_{(n,0)}\oplus \mbC))$. That is, all $ψ$ of the form $ψ = p\cdot Ψ$ with $p$ a polynomial function on $\mfu(V_{(n,0)}\oplus \mbC)$.

\begin{thm}[see \S6]\label{thm:intro_extension_polynomial}
Let $ψ$ be a polynomial type Schwartz function on $\mfu(V_{(n,0)}\oplus \mbC)$. Then $ψ$ has a smooth transfer to $\mfs_{n+1}$. More precisely, there exists an explicit polynomial function $q$ on $\mfs_{n+1}$ such that $ϕ = q\cdot Φ$ satisfies, for all $y\in \mfs_{n+1,\rs}$,
\begin{equation}\label{eq:polynomial_transfer_Lie}
\Orb(y, ϕ) = \begin{cases}
\Orb(x, ψ) & \text{if $y$ matches an element $x\in \mfu(V_{(n,0)}\oplus \mbC)$}\\
0 & \text{otherwise.}
\end{cases}
\end{equation}
\end{thm}

In particular, Theorem \ref{thm:intro_main_Lie} allows to pin down a concrete dense subspace of $\mcS(\mfu(V_{(n,0)}\oplus \mbC))$ that consists of functions that have explicit transfers. This is slightly stronger than the existence statement of \cite{Xue} for this particular signature. 

\subsection{Extension to Lie groups}

Using Theorem \ref{thm:intro_main_Lie} and the Cayley transform, we also obtain a construction of Gaussian test functions in the group setting. Consider the real manifold
\begin{equation}\label{eq:def_S_n_plus_1}
S_{n+1} := \{γ\in \GL_{n+1}(\mbC) \mid γ \ov{γ} = 1\}.
\end{equation}
As in the infinitesimal setting, $\GL_n(\mbR)$ and $U(V)$ act by conjugation on $S_{n+1}$ and $U(V\oplus \mbC)$. Again, there is a bijective matching of regular semi-simple orbits
\begin{equation}\label{eq:intro_matching_group}
[S_{n+1}]_\rs\ \simlr\ \coprod_{r + s = n} [U(V_{(r,s)}\oplus \mbC)]_\rs,
\end{equation}
and there are also orbital integrals as in the Lie algebra setting. Then we prove the following result.
\begin{thm}[see Theorem \ref{thm:main_group}]\label{thm:intro_main_group}
Let $ψ\in C^\infty(U(V_{(n,0)}\oplus \mbC))$ be an algebraic function; for example, a matrix coefficient or the character of a finite dimensional representation. Then $ψ$ has a smooth transfer to $S_{n+1}$. That is, there exists a Schwartz function $ϕ\in \mcS(S_{n+1})$ such that for every regular semi-simple $γ\in S_{n+1,\rs}$,
\begin{equation}\label{eq:polynomial_transfer_group}
\Orb(γ, ϕ) = \begin{cases}
\Orb(δ, ψ) & \text{if $γ$ matches an element $δ \in U(V_{(n,0)}\oplus \mbC)$}\\
0 & \text{no such $δ$ exists.}
\end{cases}
\end{equation}
\end{thm}

The case $ψ = 1$, which precisely states the existence of Gaussian test functions in the group setting, has to be proved first (Theorem \ref{thm:main_group}). The extension to all algebraic functions then works as in the Lie algebra case. 

Theorem \ref{thm:intro_main_group} was already known before by \cite[Proposition 4.11]{BLZZ}. What is new about our proof is that it does not require any representation-theoretic inputs from the spectral side, and that the resulting smooth transfers have a concrete description.

\subsection{Strategy of proof}

Our main task is to prove Theorem \ref{thm:intro_main_Lie}. Let $X = \GL_n(\mbR) / O(n)$ denote the symmetric space attached to $\GL_n(\mbR)$, and let $D$ be the symmetric space attached to $SO(\mfs_{n+1})$. The map $ρ$ from \eqref{eq:intro_rho} descends to a closed immersion
$$α\colon X \lr D.$$
We then analyse the intersection behaviour of $α(X)$ with the cycles $D_y$ introduced by Kudla--Millson, $y\in \mfs_{n+1,\rs}$. On the one hand, taking into account orientations, this intersection number is $ε(y)$ or $0$, depending on whether $y$ matches to signature $(n,0)$ or not (Proposition \ref{prop:intersection_topological}).

On the other hand, the Kudla--Millson form is, in a certain sense, dual to the cycles $D_y$. This is best made precise in terms of the Mathai--Quillen formalism from Branchereau's article \cite{Branchereau}. Once the required convergence statement is established, we obtain that the integrals of the Kudla--Millson form over $X$ are equal to the above intersection numbers (Theorem \ref{thm:integral_of_psi_y}). Translating this into a statement about orbital integrals completes the proof.

Two points of this construction seem miraculous to us. The first is the numerical coincidence that underlies the definition of $Φ$ via \eqref{eq:intro_key}: the negative part of the signature of $\mfs_{n+1}$ agrees with the dimension of the symmetric space $X$.

The second is the signs that are hidden in the construction. Namely, the differential forms of Kudla--Millson transform with a certain quadratic character. Specialized to our situation, we obtain that $α^*(φ_\KM)$ is $(O(n), η^n)$-invariant. At the same time, the character of $O(n)$ acting on $\det(\Sym_n(\mbR))$ is $η^{n-1}$. Taken together, this implies that $Φ$ is $(O(n), η)$-invariant which aligns with the character in Jacquet--Rallis's orbital integrals \eqref{eq:intro_def_orb_int_Lie}.

\subsection{Acknowledgements}

We heartily thank Michael Rapoport and Wei Zhang for helpful discussions during the writing of this manuscript and for their continued interest. We further thank the American Institute of Mathematics, the Max Planck Institute for Mathematics in Bonn, and the Morningside Center of Mathematics for their hospitality. Parts of this work were done during visits at these institutions.

SS is supported by a Discovery Grant from the National Science and Engineering Research Council (NSERC) of Canada. TY is partially supported by UW-Madison's Kellet Mid-Career award.

\section{Setting}

We give precise definitions of orbit matching and of the orbital integral in Definition \ref{def:intro_Gaussian_Lie}.

\subsection{Orbits and matching}
\label{ss:matching}

A reference for the following statements is \cite[\S2.1 and \S2.2]{Chaudouard}. Set $G = \GL_n(\mbR)$. An element $y\in \mfs_{n+1}$ is \emph{regular semi-simple} if the stabilizer $G_y$ is trivial and the orbit $G\cdot y \subset \mfs_{n+1}$ Zariski closed. A concrete characterization of this property is as follows. Write $y$ in a block matrix form
\begin{equation}\label{eq:block_y}
y = \begin{pmatrix}
y_0 & v \\ w & d
\end{pmatrix}\ \in\ \begin{pmatrix}
\mfs_n & i\cdot \mbR^n \\ i \cdot \mbR_n & i\cdot \mbR
\end{pmatrix}.
\end{equation}
Here and in the following, we use $\mbR^n$ and $\mbR_n$, as well as $\mbC^n$ and $\mbC_n$, to respectively denote column and row vectors. Then
\begin{equation}\label{eq:reg_ss_equivalent}
y\ \text{regular semi-simple}\quad \Longleftrightarrow\quad \mbC[y_0]\cdot v = \mbC^n\text{ and }w\cdot \mbC[y_0] = \mbC_n.
\end{equation}
The invariant of $y$ is defined as the tuple
\begin{equation}\label{eq:def_invariant_s}
\Inv(y) := \big(\mr{char}(y_0; T),\ wv,\ wy_0v,\ \ldots,\ wy_0^{n-1}v,\ d\big).
\end{equation}
It can be understood as the $\mbR$-point of the categorical quotient $G\,\back\!\back\,\mfs_{n+1} \iso \mbA^{2n+1}$ defined by $y$. It is known, see \cite[Lemma 2.1.5.1]{Chaudouard}, that two regular semi-simple elements $y$ and $y'$ satisfy
\begin{equation}\label{eq:invariant_determines_orbit}
G\cdot y = G\cdot y'\quad \Longleftrightarrow\quad \Inv(y) = \Inv(y').
\end{equation}

We turn to the unitary side. Let $V$ be an $n$-dimensional hermitian $\mbC$-vector space. An element $x\in \mfu(V\oplus \mbC)$ is \emph{regular semi-simple} if its stabilizer $U(V)_x$ is trivial and if the orbit $U(V)\cdot x\subseteq \mfu(V\oplus \mbC)$ is Zariski closed. Write $x$ in block matrix form
\begin{equation}\label{eq:block_x}
x = \begin{pmatrix}
x_0 & u \\ -u^* & d
\end{pmatrix}\ \in\ \begin{pmatrix}
\mfu(V) & V \\ V^* & i\cdot \mbR
\end{pmatrix}
\end{equation}
where we identified $\Hom(\mbC, V)\simto V$ in the obvious way. Similarly to before, $x$ is regular semi-simple if and only if $\mbC[x_0]\cdot u = V$. The invariant of $x$ is defined as
\begin{equation}\label{eq:def_invariant_u}
\Inv(x) := \big(\mr{char}(x_0; T),\ -(u,u),\ -(u,x_0u),\ \ldots,\ -(u,x_0^{n-1}u),\ d\big).
\end{equation}
As before, two regular semi-simple elements $x$ and $x'$ lie in the same $U(V)$-orbit if and only if $\Inv(x) = \Inv(x')$.
\begin{defn}\label{def:matching}
We denote by $\mfs_{n+1,\rs}$ and $\mfu(V\oplus \mbC)_\rs$ the subsets of regular semi-simple elements. Two elements $y\in \mfs_{n+1,\rs}$ and $x\in \mfu(V\oplus \mbC)_\rs$ are said to \emph{match} if $\Inv(y) = \Inv(x)$.
\end{defn}
Recall that $V_{(r,s)}$ is our notation for a hermitian $\mbC$-vector space of signature $(r,s)$. For concreteness, we choose $V_{(r,s)} = \mbC^n$ with standard hermitian form $\diag(1_r, -1_s)$ in the following. We have already stated in \eqref{eq:intro_matching_Lie} that matching defines a bijection between regular semi-simple orbits on the general linear and unitary side, see \cite[Propositions 2.1.5.2 and 2.2.4.1]{Chaudouard}.
\begin{lem}\label{lem:match_positive_definite}
A regular semi-simple element $y\in \mfs_{n+1}$ matches to signature $(n,0)$ if and only if $y$ has an orbit representative of the form
\begin{equation}\label{eq:y_diagonal}
i\cdot \begin{pmatrix}
λ_1 & & & μ_1\\
& \ddots & & \vdots\\
& & λ_n & μ_n\\
μ_1 & \cdots & μ_n & d
\end{pmatrix}.
\end{equation}
\end{lem}
We remark that an element of the form \eqref{eq:y_diagonal} is regular semi-simple if and only if the $λ_j$ are pairwise different and the $μ_j$ all non-zero. This is clear from \eqref{eq:reg_ss_equivalent}.
\begin{proof}
Assume that $y\in \mfs_{n+1}$ and $x\in \mfu(V_{(n,0)}\oplus \mbC)$ are regular semi-simple and matching. We use the notation of \eqref{eq:block_y} and \eqref{eq:block_x}.  Since $V_{(n,0)}$ is definite, and  $x \in  \mfu(V_{(n,0)})$ is regular semi-simple, the characteristic polynomial of $x_0$ is separable with eigenvalues in $i\cdot \mbR$. By the definition of matching, $\mr{char}(y_0;T) = \mr{char}(x_0; T)$. Thus $y_0$ is $G$-diagonalizable which means that $y$ is $G$-conjugate to an element of the form
\begin{equation}\label{eq:y_diagonal_intermediate}
y' = i\cdot \begin{pmatrix}
λ_1 & & & μ_1\\
& \ddots & & \vdots\\
& & λ_n & μ_n\\
μ_1' & \cdots & μ_n' & d
\end{pmatrix}.
\end{equation}
None of the $μ_k$ and $μ_k'$ vanishes by \eqref{eq:reg_ss_equivalent}. The $\mbC$-algebras $\mbC[y'_0]$ and $\mbC[x_0]$ are isomorphic via $y'_0\mapsto x_0$. For $k = 1,\ldots,n$, let $π_k(y'_0) \in \mbC[y'_0]$ denote the idempotent for the $λ_k$-eigenspace. Then
$$\begin{aligned}
iμ_k'\cdot iμ_k & = (iμ_1',\ldots, iμ_n')\cdot π_k(y'_0)\cdot {}^t (iμ_1,\ldots,iμ_n)\\[1mm]
& = -(u, π_k(x_0)u) < 0,
\end{aligned}$$
where the second equality comes from the definition of matching, and where the inequality comes from the definiteness of $V_{(n,0)}$. So
$$μ_k\cdot μ_k' = (u, π_k(x_0)u) > 0.$$
Conjugating \eqref{eq:y_diagonal_intermediate} by the diagonal matrix $\diag(|μ_1'/μ_1|^{1/2}, \ldots, |μ_n'/μ_n|^{1/2})$ brings it into the form of \eqref{eq:y_diagonal}.

Conversely, assume that $y$ has the form \eqref{eq:y_diagonal}. Set $x = y$ and view it as an element of $M_{n+1}(\mbC)$. Since $x$ satisfies ${}^t\ov{x} = -x$, it lies in the Lie algebra $\mfu(V_{(n,0)}\oplus \mbC)$. The definition of $\Inv(x)$ does not depend on whether we view it as element of $\mfs_{n+1}$ or of $\mfu(V_{(n,0)}\oplus \mbC)$. So this shows that $y$ matches to signature $(n,0)$ as claimed.
\end{proof}

\begin{lem}\label{lem:connected_components}
Consider the open subset
\begin{equation}\label{eq:def_match_n0}
\mfS = \{y\in \mfs_{n+1,\rs}\mid \text{$y$ matches to signature $(n,0)$}\}.
\end{equation}
Then $\mfS$ has two connected components which are interchanged by any $g\in G$ with $\det(g) < 0$.
\end{lem}
\begin{proof}
By Lemma \ref{lem:match_positive_definite}, every element of $\mfS$ has a representative of the form \eqref{eq:y_diagonal}. Acting with permutation matrices from $G$, and with diagonal matrices of the form $\diag(ε_1,\ldots, ε_n)$, $ε_k\in \{\pm 1\}$, we may even find such a representative with $λ_1> \ldots >λ_n$ and all $μ_k>0$. Thus, if we let $\mfR\subseteq \mfs_{n+1}$ denote the set of all matrices of the form \eqref{eq:y_diagonal} that satisfy these conditions, then
$$\mfR \subset \mfS\quad\text{and}\quad \mfS = G\cdot \mfR.$$
The manifold $\mfR$ is connected, the group $G$ has two connected components, and regular semi-simple elements have trivial stabilizer. It follows that $\mfS$ has two connected components which are interchanged by elements with negative determinant as claimed.
\end{proof}

\subsection{Orbital integrals}
\label{ss:orb_int_Lie}

Recall that $η\colon G\to \{\pm 1\}$ denotes the sign character. The transfer factors used in \cite{JR, BLZZ} are adapted to a global trace formula setting which we do not need here, so we use the following simple definition:

\begin{defn}  \label{def:transfer factor}
A \emph{transfer factor} is a locally constant function
$$ε\colon \mfs_{n+1,\rs}\lr \{\pm 1\},$$
that satisfies $ε(g^{-1}yg) = η(g)ε(y)$ for all $g \in G$ and $y \in \mfs_{n+1,\rs}$.
\end{defn}

Note that the orbital integrals of Gaussian test functions are non-zero only for $y\in \mfS$, and that $ε\vert_\mfS$ is unique up to sign by Lemma \ref{lem:connected_components}. So for the purposes of our article, choosing $ε$ can be understood as fixing one of the two $(G,η)$-invariant sign functions on $\mfS$.

We fix a transfer factor $ε$ for the rest of the article. We also fix a Haar measure for $G$.

\begin{defn}\label{def:orbital_integral}
Let $Φ\in \mcS(\mfs_{n+1})$ be a Schwartz function and $y\in \mfs_{n+1}$ a regular semi-simple element. The orbital integral $\Orb(y, Φ)$ is defined by
\begin{equation}\label{eq:def_orb_int_raw}
\Orb(y, Φ) := ε(y) \int_G Φ(g^{-1} y g) η(g)\ dg.
\end{equation}
It only depends on the orbit $G\cdot y$.
\end{defn}

\section{Intersection numbers}

\subsection{Symmetric spaces} \label{ss:symmetric spaces}

Let $G^0 = \GL_n(\mbR)^{\det > 0}$ be the identity connected component. Let $K = SO(n)\subset G^0$ be the standard maximal compact subgroup and let $X = G^0/K$ be the corresponding symmetric space.

We denote by $\Sym_n$ and $\Skew_n$ the real vector spaces of symmetric (resp. skew-symmetric) $(n\times n)$-matrices. We write $\Sym_n^{>0}$ for the positive definite symmetric matrices. We can describe $X$ as
\begin{equation}\label{eq:def_sym_GL_n}
X \simlr \Sym_n^{>0},\quad gK \longmapsto {}^tg^{-1}\cdot g^{-1}.
\end{equation}

Recall that we endowed $\mfs_{n+1}$ with the quadratic form $Q(y) = -\tr(y^2)$. There is an orthogonal decomposition
$$\mfs_{n+1} = (i\cdot \Sym_{n+1}) \overset{\perp}{\oplus} (i\cdot \Skew_{n+1})$$
where the quadratic form is positive definite on the first, and negative definite on the second summand. Let
$$K_{SO} = SO(i\cdot \Sym_{n+1})\times SO(i\cdot \Skew_{n+1})$$
be the corresponding maximal compact subgroup of the identity connected component $SO(\mfs_{n+1})^0$, and let
$$D := SO(\mfs_{n+1})^0\,/\,K_{SO}$$
be the quotient symmetric space. A concrete description of $D$ is given by
\begin{equation}\label{eq:embed_symmetric}
\begin{aligned}
D & \ \simlr\ \{W \subseteq \mfs_{n+1}\mid Q\vert_W < 0,\ \dim(W) = n(n+1)/2\}\\[1mm]
hK_{SO} & \ \longmapsto\ h\cdot (i\cdot \Skew_{n+1}).
\end{aligned}
\end{equation}
Recall that $\rho \colon G  \lr SO(\mfs_{n+1})$ denotes the orthogonal representation that defines the action of $G$ on $\mfs_{n+1}$, i.e.\ for $g \in G$ and $y \in \mfs_{n+1}$, we set
\begin{equation} \label{eqn:rho def}
\rho(g) y =    \begin{pmatrix} g & \\ & 1 \end{pmatrix} y \begin{pmatrix} g^{-1} & \\ & 1 \end{pmatrix} .
\end{equation}
To simplify notation, we will often write $g \cdot y$ for $\rho(g) y$. The map $\rho$ descends to a closed immersion
\begin{equation}\label{eq:def_alpha}
α\colon X\lr D
\end{equation}
of real manifolds, and our next aim is to describe it in terms of \eqref{eq:def_sym_GL_n} and \eqref{eq:embed_symmetric}.

Suppose that $H \in M_n(\mbR)$ is a symmetric $(n\times n)$-matrix with $\det(H) \neq 0$. Then we denote by $\Sym(H)$ and $\Skew(H)$ the real vector spaces of matrices that are (skew-)symmetric with respect to $H$. That is,
$$\begin{aligned}
\Sym(H) &\ =\ \{A\in M_n(\mbR)\mid {}^tA = H A H^{-1}\}\\[1mm]
\Skew(H) &\ =\ \{A\in M_n(\mbR)\mid {}^tA = -H A H^{-1}\}.
\end{aligned}$$
Observe that for $g\in \GL_n(\mbR)$,
\begin{equation}\label{eq:trafo_symmetric}
g\cdot \Sym(H)\cdot g^{-1} = \Sym({}^tg^{-1} \cdot H\cdot g^{-1}),
\end{equation}
and analogously for $\Skew(H)$. With this terminology in place, $α$ is given by
\begin{equation}\label{eq:alpha_explicit}
\begin{aligned}
α\colon\Sym_n^{>0} & \ \lr \ D\\[1mm]
H & \ \longmapsto\ i\cdot \Skew\left(\left(\begin{smallmatrix} H & \\ & 1\end{smallmatrix}\right)\right).
\end{aligned}
\end{equation}

\subsection{Kudla--Millson cycles}
\label{ss:KM_cycles}

For a vector $y \in \mfs_{n+1}$, Kudla and Millson introduced the totally geodesic submanifold
$$D_y = \{W\in D\ \mid\ y\perp W\}.$$
Clearly, as long as $y\neq 0$,
$$
D_y = \begin{cases}
\text{symmetric space for $SO(\langle y\rangle^\perp)$} & \text{if $Q(y) > 0$}\\
\emptyset & \text{if $Q(y) \leq 0$}.
\end{cases}$$
In particular, $D_y$ is of codimension $n(n+1)/2$ if $Q(y) >0$.
\begin{prop}\label{prop:intersection_set}
Let $y\in \mfs_{n+1}$ be regular semi-simple. Then
\begin{equation}\label{eq:intersection_point_number}
|X\cap D_y| = \begin{cases}
1 & \text{if $y$ matches to signature $(n,0)$}\\[1mm]
0 & \text{otherwise.}\end{cases}
\end{equation}
Moreover, in the first case, the intersection $X\cap D_y$ is transversal.
\end{prop}
\begin{proof}
Assume that $X\cap D_y$ is non-empty. This means that there exists a positive definite quadratic form $H \in \Sym_n$ such that $y\perp i\cdot \Skew\left(\left(\begin{smallmatrix} H & \\ & 1\end{smallmatrix}\right)\right)$, which is equivalent to $y\in i\cdot \Sym\left(\left(\begin{smallmatrix} H & \\ & 1\end{smallmatrix}\right)\right)$. In terms of the block coordinates from \eqref{eq:block_y}, this implies that $y_0$ is diagonalizable with eigenvalues in $i\cdot \mbR$. Put differently, the $G$-orbit $G\cdot y$ contains an element $y'$ of the form \eqref{eq:y_diagonal_intermediate}:
\begin{equation}\label{eq:y_diagonal_intermediate_repeat}
y' = i\cdot \begin{pmatrix}
λ_1 & & & μ_1\\
& \ddots & & \vdots\\
& & λ_n & μ_n\\
μ_1' & \cdots & μ_n' & d
\end{pmatrix}.
\end{equation}
Since $X\cap D_y \simto X\cap D_{y'}$, there exists a positive definite symmetric matrix $H' \in X$ such that $y' \in i\cdot \Sym\left(\left(\begin{smallmatrix} H' & \\ & 1\end{smallmatrix}\right)\right)$. Since $y'$ is regular semi-simple, the eigenvalues $λ_1,\ldots,λ_n$ are pairwise different. Hence, the only possibility for $H'$ is to be diagonal. It is then uniquely determined by \eqref{eq:y_diagonal_intermediate_repeat} as $H' = \mr{diag}(μ_1'/μ_1,\ldots,μ_n'/μ_n)$. Since $H'$ is positive definite, this implies that the ratios $μ_k'/μ_k$ are strictly positive, meaning $μ_k'$ and $μ_k$ have the same sign for all $k = 1,\ldots,n$. Hence $y'$ is $G$-conjugate to an element of the form \eqref{eq:y_diagonal} and hence matches to signature $(n,0)$ by Lemma \ref{lem:match_positive_definite}. The argument also showed that $X\cap D_y$ is a singleton whenever it is non-empty.

Conversely, assume that $y$ matches to signature $(n,0)$. (This in particular implies that $Q(y) > 0$.) By the invariance property $g\cdot D_y = D_{g\cdot y}$, we may assume that $y$ is of the diagonal form \eqref{eq:y_diagonal}. In particular $y\in i\cdot \Sym_{n+1}$, which is equivalent to $y\perp i\cdot \Skew_{n+1}$. In terms of \eqref{eq:alpha_explicit}, this means that $H = 1_n$ lies in $X\cap D_y$. This proves the set-theoretic statement \eqref{eq:intersection_point_number}.

It is left to show the transversality of the intersection. Assume that there exists a non-zero tangent vector to a point $e\in X\cap D_y$,
$$0\neq v\in T_e(X) \cap T_e(D_y)$$
where the intersection is taken in the tangent space $T_e(D)$. Since both $X$ and $D_y$ are totally geodesic submanifolds of $D$, the geodesic through $e$ defined by $v$ is contained in both $X$ and $D_y$, in contradiction with the fact that $|X\cap D_y| = 1$.
\end{proof}

\subsection{Orientations}
\label{ss:orientations}
We fix an orientation on $D$, as well as an orientation on the vector space $ V^- := i\cdot \Skew_{n+1} \subset \mfs_{n+1}$. As described in \cite[\S 2]{KM}, these choices induce an orientation on each submanifold $D_y$ for $y \in \mfs_{n+1}$. Note that if $V$ is a quadratic space of signature $(p,q)$ with $p\geq 2$ and $q\geq 1$, then the set of positive length vectors $V_{>0} \subset V$ forms a connected manifold. Thus there are only two conventions for orienting the family $\{D_y\}_{y\in\mfs_{n+1},\,y\neq 0}$ such that the orientations vary continuously in $y$. The definition in \cite[\S2]{KM} can be understood as fixing one of them.

Given an orientation on $X$, let $[X]$ denote the resulting oriented manifold. The transversality statement in Proposition \ref{prop:intersection_set} allows to consider the topological intersection numbers $[X]\cdot_{[D]}[D_y] \in \{\pm 1\}$ whenever $y$ is regular semi-simple. Recall that we fixed the transfer factor $ε$.

\begin{prop}\label{prop:intersection_topological}
There exists an orientation on $X$ such that for every regular semi-simple element $y\in \mfs_{n+1}$,
\begin{equation}\label{eq:top_intersection}
[X]\cdot_{[D]}[D_y] = \begin{cases}
ε(y) & \text{if $y$ matches to signature $(n,0)$}\\[1mm]
0 & \text{otherwise.}
\end{cases}
\end{equation}
\end{prop}
\begin{proof}
Let $z\in \mfs_{n+1,\rs}$ match to signature $(n,0)$. Then $X$ intersects $D_z$ by Proposition \ref{prop:intersection_set}. Fix the orientation on $X$ such that $[X]\cdot_{[D]}[D_z] = ε(z)$. We claim that, with this choice, \eqref{eq:top_intersection} holds for all $y\in \mfs_{n+1,\rs}$.

By Proposition \ref{prop:intersection_set}, the intersection number $[X]\cdot_{[D_y]}[D]$ is non-zero precisely for $y\in \mfS$, the subset of $y\in \mfs_{n+1,\rs}$ that match to signature $(n,0)$. This set has two connected components by Lemma \ref{lem:connected_components}, and these are interchanged by $G\setminus G^0$. By definition of the transfer factor, we have $ε(g\cdot y) = η(g)ε(y)$ for all $g\in G$. Our task is hence to show that we have the identity
$$[X]\cdot_{[D]} [D_{g\cdot y}] = η(g)\cdot [X]\cdot_{[D]} [D_y]$$
for one choice of $y\in \mfS$.

The action of $G^0$ on $X$ and $D$ extends to an action of $G$ by extending the formulas in \eqref{eq:def_sym_GL_n} and \eqref{eq:embed_symmetric}:
$$\begin{aligned}
G\times X & \lr X,\quad g\cdot H = {}^tg^{-1}\cdot H g^{-1}\\
G\times D & \lr D,\quad g\cdot W = gWg^{-1}.
\end{aligned}$$
For all $g\in G$, invariance of intersection numbers under isomorphisms implies that
$$[X]\cdot_{[D]}[D_y] = (g\cdot [X])\cdot_{(g\cdot [D])} (g\cdot [D_y])$$
where the terms $(g\cdot [M])$ denote the image $g(M)$ together with their pushforward orientation.

We consider the specific element $σ = \diag(-1,1,\ldots,1) \in G$ which stabilizes the base point $e = 1_n\in X$. It is clear from definitions that $σ\cdot D_y = D_{σy}$. We need to understand how $σ$ interacts with the orientations on $X$, $D$, $D_y$ and $D_{σy}$.

\medskip \noindent (1) The tangent space $T_e(X)$ is $\Sym_n$ and conjugation by $σ$ on $\Sym_n$ has determinant $(-1)^{n-1}$. So we obtain
\begin{equation}\label{eq:sign_X}
σ\cdot [X] = (-1)^{n-1}[X].
\end{equation}

\medskip \noindent (2) The tangent space $T_e(D)$ is $\Hom(i\cdot \Skew_{n+1}, i\cdot \Sym_{n+1})$. Conjugation by $σ$ acts with determinant
$$((-1)^n)^{\dim(\Skew_{n+1})}\cdot ((-1)^n)^{\dim(\Sym_{n+1})} = (-1)^{n(n+1)(n+1)} = 1.$$
So we find
\begin{equation}\label{eq:sign_D}
σ\cdot [D] = [D].
\end{equation}

\medskip \noindent (3) Finally, let $y\in \mfS$ be such that $e\in D_y$. The orientation on $T_e(D_y)$ is defined as follows. Recall that we have fixed a reference orientation on $i\cdot \Skew_{n+1}$. Orient the line $\langle y\rangle$ by declaring $y$ to be positive. Consider the induced orientation on $\Hom(i\cdot \Skew_{n+1}, \langle y\rangle)$. Finally, orient the space $T_e(D_y) = \Hom(i\cdot \Skew_{n+1}, \langle y \rangle^\perp)$ by requiring that the direct sum decomposition
\begin{equation}\label{eq:tangent_decomposition_orientation}
\Hom(i\cdot \Skew_{n+1}, i\cdot \Sym_{n+1}) = T_e(D_y) \oplus \Hom(i\cdot \Skew_{n+1}, \langle y\rangle)
\end{equation}
is compatible with orientations. Moreover, $\sigma$ acts on $i\cdot \Skew_{n+1}$ via a transformation with determinant $(-1)^n$. Thus, by (2) and \eqref{eq:tangent_decomposition_orientation}, we conclude
\begin{equation}\label{eq:sign_D_y}
σ\cdot [D_y] = (-1)^n [D_{σy}].
\end{equation}
Taking \eqref{eq:sign_X}, \eqref{eq:sign_D} and \eqref{eq:sign_D_y} together, we obtain
\begin{equation}
[X]\cdot_{[D]}[D_{σy}] = - [X]\cdot_{[D]} [D_{y}]
\end{equation}
as we needed to show.

\end{proof}
\section{Schwartz functions}
\subsection{Mathai--Quillen formalism}
\label{section:MQ}
We begin by recalling a general formalism of Mathai and Quillen \cite{MQ}. Our presentation here follows \cite{BGV}. Assume that $M$ is an oriented manifold, and let $C^{\infty}_M$  (resp.\ $\Omega^{\bullet}_M$) denote the vector bundles of smooth functions (resp.\ differential forms) over $M$. Tensor products in the following are meant as smooth vector bundles over $M$, meaning over $C^\infty_M$.

Suppose that  $(E, (\ ,\ ))$ is an oriented metrized vector bundle of rank $r$ over $M$, and that $\nabla\colon E\to \Omega^1_M\tensor E$  is a connection that is compatible with the metric in the sense that
$$d(s_1,s_2) = (\nabla s_1, s_2) + (s_1, \nabla s_2).$$
Let $κ \colon E\to \Omega^2_M\tensor E$ denote the curvature, which is a section of $\Omega^2_M\tensor \mcE nd(E)$. From the compatibility of $\nabla$ with the metric it is immediate that
$$(κ(s_1), s_2) + (s_1, κ(s_2)) = 0$$
for all sections $s_1,s_2 \colon M\to E$. This means that $κ$ defines a section $κ:M\to \Omega^2_M\tensor \mfs\mfo(E)$. We identify $\wedge^2 E \simto \mathfrak{so}(E)$ by
$$s_1 \wedge s_2 \longmapsto [v \mapsto (s_1, v) s_2 - (s_2, v) s_1]$$ and view $κ$ as a section
\begin{equation}
κ \colon M\lr \Omega^2_M\tensor (\wedge^2 E).
\end{equation}
Consider the bigraded (non-commutative) algebra $\mcA = \Omega^\bullet_M \tensor (\wedge^\bullet E)$ with sign convention
$$(σ_1 \tensor τ_1)\cdot (σ_2\tensor τ_2) = (-1)^{\deg(τ_1)\deg(σ_2)} (σ_1\wedge σ_2) \tensor (τ_1\wedge τ_2).$$
The last piece of notation we need is the \emph{Berezin integral}
$$\{-\}\colon\mcA \lr \Omega^\bullet_M.$$
It is given by composing the projection $\mcA\to \Omega^\bullet_M \tensor (\wedge^\mr{top} E)$ with the (unique) trivialization $(\wedge^\mr{top}E) \simto M \times \mathbb R$ that preserves orientation and metric.

Now suppose that $s\colon M\to E$ is a section. Consider the three elements
\begin{equation}
\begin{aligned}
|s|^2 :=	(s, s) & \in C^\infty_M,\\[1mm]
\nabla(s)& \in \Omega_M^1\tensor E\\[1mm]
κ & \in \Omega_M^2 \tensor (\wedge^2E).
\end{aligned}
\end{equation}
all viewed in the algebra $\mathcal A$.
\begin{defn}\label{def:MQ_form} Let $r = \mathrm{rank}(E)$.
The \emph{Mathai--Quillen form} of $s$ defined by the above data is the $r$-form
$$ψ_s := (-1)^{r(r-1)/2}  (2 \pi)^{-r/2}\big\{e^{-2 \pi |s|^2 - 2 \sqrt \pi \nabla(s) - κ}\big\} \ \in \Omega^r_M.$$
Here the exponential is defined by the usual power series, with products taking place in the algebra $\mathcal A$.
We remark that the normalizing factors are motivated by Theorem \ref{thm:Brancherau} below.
\end{defn}

We denote by $Z(s)$ or $Z_s\subseteq M$ the zero locus of $s$. Assume in addition that $s$ is a \emph{regular} section, meaning that for all points $x\in Z_s$, the induced map $ds \colon T_x(M)\to E_x$ from tangent space to fiber is surjective. Then $Z_s$ is a submanifold of $M$ of codimension $r$.

\begin{defn} \label{def:orientation_on_Z_s}
Let $[Z_s]$ denote the submanifold $Z_s$ equipped with the following orientation: for any point $x \in Z_s$, the derivative $ds$ defines an isomorphism
$N_x \simeq E_x$ where $N$ is the normal bundle of $Z_s$. The fixed orientation on $E$ pulls back to an orientation of $N$. As we have also fixed an orientation on $M$, we define an orientation on $Z_s$ via the identification $\wedge^{top} N_x \otimes \wedge^{top} T_x Z_s \simto \wedge^{top} T_x(M)$.
\end{defn}

For a compactly supported differential form $\eta \in \Omega^{\bullet}_{c,M}$, we define the $\delta$-current
\[
\delta_{[Z_s]}(\eta) :=  \int_{[Z_s]} \eta\vert_{Z_s}.
\]

\begin{prop} \label{prop:limiting current}
Let $M$ be an oriented manifold, $E$ an oriented metrized vector bundle of rank $r$, and $\nabla$ a compatible connection. Suppose that $s$ is a regular section of $E$ with oriented vanishing locus $[Z_s]$.

Let $\psi_s$ denote the Mathai--Quillen form, and for a compactly supported differential form $\eta \in \Omega^{\bullet}_{c,M}$ consider the current
\[
[\psi_s] (\eta) := \int_M \psi_s \wedge \eta.
\]
Note that for $t \in \mathbb R_{>0}$, the section $ts$ is again regular, and induces the same orientation on $Z_s = Z_{ts}$. We then have
\[
\lim_{t \to \infty} [\psi_{ts}] =  \delta_{[Z(s)]}.
\]
as currents on $M$.
\begin{proof}
This proposition follows from the estimates in \cite[Theorem 3.12]{BGS-Grothendieck}; see also \cite[Theorem 2.1]{Getzler} for a direct proof in local coordinates.
\end{proof}
\end{prop}

We will also  need the following ``transgression formula'':
\begin{prop} \label{prop:transgression}
Let $M$ be an oriented manifold and $E$ an oriented metrized vector bundle, equipped with a compatible connection, as above. For any section $s$ of $E$, we define the \emph{transgression form}
\[
\zeta_s := (-1)^{r(r-1)/2} (2\pi)^{-r/2} \, \left\{ s \wedge e^{ - 2 \pi |s|^2 - 2 \sqrt{\pi} \nabla(s) - \kappa } \right\} \in \Omega^{r-1}(M).
\]
Then for $t \in \mathbb R_{>0}$, we have
\[
t \frac{\partial}{\partial t} \psi_{ts} = - 2 \sqrt{\pi}  \, d \zeta_{ts}.
\]
\begin{proof}
This is proved in \cite[\S 7]{MQ}; see also \cite[Prop.\ 1.3]{Getzler}.
\end{proof}
\end{prop}

\begin{prop} \label{prop: Green current} Suppose $M$ is an oriented manifold, $E$ is an oriented vector bundle of rank $r$, and $s$ is a regular section of $E$ with  vanishing locus $[Z_s]$ oriented as per Definition \ref{def:orientation_on_Z_s}.

\medskip \noindent (1) On $M \setminus Z_s$, the integral
\begin{equation}\label{eq:def_g_s}
g_s := \int_{1}^{\infty} \zeta_{ts} \frac{dt}t
\end{equation}
defines a smooth form.

\medskip \noindent (2) The form $g_s \in \Omega^{r-1}(M \setminus Z_s)$ extends to a  locally $L^1$-form on $M$.

\medskip \noindent (3) Let $[g_s]$ denote the  current given by integration against $g_s$. Then we have the identity
\begin{equation}\label{eq:current_equation}
2 \sqrt \pi \,  d [g_s]  + \delta_{[Z_s]} = [ \psi_s]
\end{equation}
of currents on $M$.
\end{prop}
\begin{proof} These statements follow from analogous estimates to those found in \cite[Theorem 3.12]{BGS-Grothendieck}, which are presented in greater generality (and under slightly different assumptions) in \emph{loc.\ cit.} For the convenience of the reader, we give a self-contained argument here.

Let  $C = (-1)^{r(r-1)/2} (2 \pi)^{-r/2}$.
Since $|s|^2$ commutes with both $\nabla(s)$ and $\kappa$ in the algebra $\mathcal A$, we may write
\begin{equation} \label{eqn:zeta s extract exp}
\zeta_s = C e^{-2 \pi |s|^2}\{ s \wedge e^{ - 2 \sqrt \pi \nabla(s) - \kappa} \}
\end{equation}
Note that the expression $\{ s \wedge e^{ - 2 \sqrt \pi \nabla(s) - \kappa} \}$ is a   polynomial expression in $s, \nabla(s)$ and $\kappa$. In particular, we may write
\[
\zeta_s = \sum_{k=1}^{r} \eta_{k}(s) e^{- 2\pi |s|^2}
\]
where $\eta_k(s) \in \Omega^{r-1}(M)$ is homogeneous in $s$ of degree $k$, i.e.\ $\eta_k(ts) = t^k \eta_k(s)$ for $t \in \mathbb R$. Then, on $M \setminus Z_s$, we have
\begin{equation}	\label{eqn:gs homogeneous}
\begin{split}
g_s &=  \sum_{k=1}^{r} \left(  \int_1^{\infty} t^k e^{-2 \pi t^2 |s|^2} \frac{dt}{t} \right) \eta_k(s) \\
&= \sum_{k = 1}^{r} |s|^{-k} \left(  \int_{|s|}^{\infty} t^k e^{-2 \pi t^2} \frac{dt}{t} \right) \eta_k(s) .
\end{split}
\end{equation}
For each $k>0$, the integrals appearing above define smooth functions on $M \setminus Z_s$
and the function $|s|^{-k}$ is also smooth on $M \setminus Z_s$.
Part (1) follows from these observations.

To prove part (2), it suffices to show that $g_s$ is locally $L^1$ in a neighbourhood of $Z_s$. If $Z_s$ is empty, there is nothing to show. Otherwise, fix a point $z \in Z_s$ and a coordinate chart $U \subset M$ around $z$ with coordinates $x_1, \dots x_n$ mapping $z$ to $0 \in \mathbb R^n$.
We may further assume that there is a local orthonormal frame $e_1, \dots, e_r$ of $E|_U$ such that
\[ s = \sum_{i = 1}^r x_i e_i. \]
In particular, we have $	|s|^2 = x_1^2 + \cdots + x_r^2.$
Substituting the above expression for $s$ into \eqref{eqn:zeta s extract exp}, we conclude that upon restriction to $U$, each $\eta_k(s)$ can be written as a linear combination  of the form
\[
\eta_k(s)|_U= \sum_{i=1}^r	x_i \cdot \left( \text{smooth form on } U \right).
\]
Moreover, all the integrals appearing in the second line of \eqref{eqn:gs homogeneous} are uniformly  bounded on $M$.   We are reduced to showing that the functions
\[
\frac{x_i}{|s|^k} = \frac{x_i}{(x_1^2 + \cdots + x_r^2)^{k/2}}, \qquad  i,k \leq r
\]
are integrable on a sufficiently small open neighbourhood of $0 \in \mathbb R^n$, which can be seen by a straightforward calculus argument via polar coordinates.

Finally, we prove part (3). Suppose $\eta$ is a compactly supported form on $M$. Then
\begin{equation} \label{eqn:gs Green eqn calc}	
\begin{split}
d[g_s](\eta) = \int_M g_s \wedge d\eta &=  \int_M \left( \lim_{t \to \infty} \int_{1}^t \zeta_{rs}\frac{dr}r  \right)  \wedge d \eta \\
&= \lim_{t \to \infty} \int_M \ \left( \int_1^t \zeta_{rs} \frac{dr}r \right) \wedge d \eta     \\
&= \lim_{t \to \infty}  \int_M d  \left( \int_1^t  \zeta_{rs}  \frac{dr}r \right) \wedge \eta
\end{split}
\end{equation}
where the interchange of limit and integral in the second equality is justified by the proof of part (2) and the dominated convergence theorem. Applying Proposition \ref{prop:transgression}, we have
\begin{align*}
d \left( \int_1^t \zeta_{rs} \frac{dr}r \right) &= \int_1^t d \zeta_{rs}\frac{dr}r \\
&= - \frac{1}{2 \sqrt{\pi}} \int_1^t \frac{\partial}{\partial r} \psi_{rs} dr \\
&= - \frac{1}{2 \sqrt{\pi}} \left( \psi_{ts} - \psi_s \right).
\end{align*}
Substituting this into \eqref{eqn:gs Green eqn calc} and applying Proposition \ref{prop:limiting current} yield the result.

\end{proof}

\subsection{The Kudla--Millson form and a result of Brancherau}

We begin with a slightly more general situation. Let $(V,Q)$ be a real quadratic space of signature $(p,q)$. Let $V = V^+ \oplus V^-$ be a decomposition into maximal positive and negative, definite subspaces respectively. Let $D(V) = SO(V)^0/ SO(V^+) \times SO(V^-)$ be the corresponding symmetric space. As before, there is an identification
\[
D(V) \simeq \left\{  z \subset V \ | \ Q|_z < 0 \text{ and } \dim z = q  \right\}.
\]

The space $D(V)$ is naturally equipped with a tautological vector bundle $ \widetilde E$ of rank $q$; concretely, the fibre $\widetilde E_z$ at a point $z \in D$ is simply the space $z$. We define a metric $\left( \cdot ,\cdot \right)_{\widetilde E} $ on $\widetilde E$ by the formula
\[
(\wt s, \wt s)_{\wt E}(z) = -2Q(\wt s(z)), \qquad \wt s\colon D(V)\to \wt E.
\]
It is clear that the datum $(\widetilde E, ( \cdot, \cdot )_{\widetilde E} ) $ is naturally $SO(V)^0$-equivariant. Given any $v \in V$, there is a section
\[
\widetilde s_v \colon D(V) \lr \widetilde E, \qquad \widetilde s_v(z) = \mathrm{pr}_z(v)
\]
i.e.\, for $z \in D(V)$, we take the orthogonal projection $\mathrm{pr}_z(v)$ of $v$ onto $z$. By construction, this section satisfies
\begin{equation} \label{eqn:s_v equivariance}
\gamma^*  \widetilde s_v = \widetilde s_{\gamma^{-1} v}
\end{equation}
for any $\gamma \in SO(V)^0$.  Moreover, we have
\[
Z(\widetilde s_v) = D_v
\]
where
\[
D_v := \{ z \in D(V) \ | \ z \perp v \}
\]
is the Kudla--Millson cycle from Section \S\ref{ss:KM_cycles}.

Furthermore, fix an orientation on $D(V)$ and $\wt E$. We note that as $D(V)$ is connected, an orientation on $\wt E$ is determined by the choice of an orientation on any single fibre $\wt E_z$. Such a choice determines an orientation of $D_y$, as in \cite[\S 2]{KM}. On the other hand, if $Q(v)>0$, then it is straightforward to check that $\wt s_v$ is regular, which in turn induces an orientation on $[Z(\wt s_y)]$ via  Definition \ref{def:orientation_on_Z_s}. Unwinding the definitions, one can verify that the two constructions coincide, i.e.
\begin{equation} \label{eqn:orientations match}
[Z(\wt s_y)] = [D_y].
\end{equation}

Finally, the bundle $\wt E$ is equipped with a natural connection $\nabla_{\wt E}$ called the \emph{Maurer-Cartan connection}; see e.g. \cite[\S 3.1]{Branchereau} for details. This connection is compatible with the metric on $\wt E$, and is $SO(V)^0$-equivariant.

Having specified the requisite data, we now have the  corresponding Mathai--Quillen form
\begin{equation}
\psi_{\wt s_v} \in \Omega^{q}(D(V))
\end{equation}
as in Definition \ref{def:MQ_form}.

On the other hand, Kudla and Millson have constructed an explicit differential form, satisfying a natural Thom form property with respect to $D_y$. We briefly recall the construction, referring to \cite[\S2]{Branchereau} for a more detailed discussion. We begin by noting that there is a canonical identification
\[
\mathfrak{P}  := 	T_{e}(D(V)) \simeq \Hom(V^-, V^+).
\]
Note also that the  linear action of $O(V)$ on $V$ induces an action on $D(V)$, and the maximal compact subgroup $\wt K := O(V^+)  \times O(V^-)$ acts on $\mathfrak P$ by post- and pre-composition in the natural way.

Fix orthonormal bases $x_1, \dots, x_p$ for $V^+$ and $x_{p+1}, \dots, x_{p+q}$ for $V^-$. These induce a basis $\{ X_{ij} \}$ for $\mathfrak P$, with $1 \leq i \leq p$ and $p+1 \leq j \leq p+q$. Concretely, $X_{ij}(x_k) = δ_{jk}x_i$. Let $\omega_{ij}$ denote the dual basis, and define the Howe operator
\[
H \colon \mcS(V) \otimes_{\mathbb C}  \wedge^{\bullet} \mathfrak P^*  \lr \mcS(V) \otimes_\mbC \wedge^{\bullet + q} \mathfrak P^*
\]
by the formula
$$H :=  2^{-q} \cdot \prod_{j = p+1}^{p+q} \sum_{i = 1}^p \left[\left((x_i - \frac{1}{2π} \frac{\partial}{\partial x_i}\right) \tensor A_{ij}\right]$$
where $\mcS(V)$ is the space of Schwartz functions on $V$,  and $A_{ij}$ is left multiplication by $\omega_{ij}$ .
\begin{defn}[\protect{\cite[\S5]{KM}}]\label{def:KM_form}
The Kudla--Millson form $φ_\KM$ is the $q$-form obtained by applying the Howe operator to the Gaussian:
$$φ_\KM := H\cdot e^{-π(\sum_{i = 1}^{p+q} x_i^2)} \in \mcS(V) \tensor (\wedge^q \mfP^*).$$
\end{defn}

This form satisfies the following equivariance property. Let $\nu \colon O(V) \to \{\pm 1\}$ denote the spinor norm. Recall that this is the unique character whose restriction to $\wt K = O(V^+)\times O(V^-)$ is given by $ν(k^+,k^-) = \det(k^-)$. We then have
\begin{equation} \label{eqn:PhiKM equivariance}
\varphi_{\KM}(k^{-1}  v ) =  \nu(k)  \cdot k^*(\varphi_{\KM}(v)) \in \wedge^q \mathfrak P^*
\end{equation}
for all $v \in V$ and $k \in \wt K$, see \cite[Theorem 3.1]{KM_harmonic_I}.

In particular, $φ_\KM$ lies in the space of invariants $(\mcS(V) \otimes \wedge^q \mathfrak P^*)^{SO(V^+) \times SO(V^-)}$. There is a natural isomorphism
\[
\left[ \mcS(V) \otimes \wedge^q \mathfrak P^*\right]^{SO(V^+) \times SO(V^-)} \simlr  \left[ \mcS(V) \otimes \Omega^q(D(V))\right]^{SO(V)^0}, \qquad \eta \longmapsto \wt \eta
\]
determined by the relation $\wt \eta|_{e} = \eta$. We define
\[
\wt \varphi_{\KM} \in  \left[ \mcS(V) \otimes \Omega^q(D(V))\right]^{SO(V)^0}
\]
to be the image of $\varphi_{\KM}$ under this isomorphism.

We now have two constructions of differential forms associated to a vector $v \in V$. The following theorem of Brancherau asserts that the two constructions coincide up to normalization:
\begin{thm}[{\cite{Branchereau}}] \label{thm:Brancherau}
For any $v \in V$, we have
\[
\wt \varphi_{\KM}(v) = 2^{-q/2}	e^{- 2\pi Q(v)} \psi_{\wt s_v} .
\]
\qed
\end{thm}

Now we specialize the notation to our case of interest. Take $V = \mathfrak s_{n+1}$ and $D = D(\mathfrak s_{n+1})$ as before. Let $X = \Sym_n^{>0}(\mathbb R)$ and consider the map $\alpha \colon X \to D$ as in \eqref{eq:def_alpha}. In fact, it will be more convenient to work on the group $B \subset G$ of upper triangular matrices.
Consider the surjective map
\begin{equation}
B \lr X = \Sym_n^{>0}, \qquad b \longmapsto {}^t b^{-1} b^{-1}
\end{equation}
which is a finite covering map of degree $2^n$, and let $\beta$ be the composition
\[
\beta \colon B \lr X \overset{\alpha}{\lr} D.
\]
Pulling back via $\beta$, we obtain the bundle  $E := \beta^* \widetilde E$,  equipped with the pullback metric $(\cdot, \cdot)_E$ and connection $\nabla_E$, as well as a pullback section $s_y = \beta^* \wt s_y$ for any $y \in \mfs_{n+1}$. We let
\begin{equation}
\psi_y := \psi_{s_y} \in \Omega^{n(n+1)/2}(B).
\end{equation}
and
\begin{equation}
\zeta_y := \zeta_{s_y} \in \Omega^{n(n+1)/2 - 1}(B)
\end{equation}
denote the Mathai--Quillen and transgression forms attached to this data. Note that $\dim(B) = n(n+1)/2$, i.e.\ $\psi_y$ is a form of top degree on $B$. Moreover, it is clear from the construction that these forms are functorial in the data defining them, i.e.\ we have
\[
\psi_y = \beta^* \psi_{\wt s_y}, \qquad \zeta_y = \beta^* \zeta_{\wt s_y}.
\]

\begin{rmk}
The role that Theorem \ref{thm:Brancherau} plays in our present work is as follows. On the one hand, as evident in Definition \ref{def:KM_form}, it is straightforward to write down explicit formulas for the Kudla--Millson form. On the other hand, we may apply the general machinery of the Mathai--Quillen formalism, in particular the current equation \eqref{eq:current_equation}, to deduce geometric properties that are not immediately evident from the construction.
\end{rmk}

\subsection{Schwartz forms}
We will employ the terminology of Schwartz forms on real affine algebraic varieties, which is a special case of that of Nash manifolds discussed in \cite{AG}.

\begin{defn}
Let $M$ be a real affine algebraic smooth variety, and let $C^{\infty}(M)$ and $\Omega^{\bullet}(M)$ denote the space of smooth functions, and smooth differential forms, respectively. A  function $f \in C^{\infty}(M)$ is called a \emph{Schwartz function} if for every algebraic differential operator $D$, the function $D f$ is bounded on $M$ (cf.\ \cite[Corollary 4.1.3]{AG}). The space of \emph{Schwartz (differential) forms} is the subspace of $\Omega^{\bullet}(M)$ spanned by elements of the form $f \omega$ where $f$ is a Schwartz function and $\omega$ is an algebraic differential form.
\end{defn}

Note that if $M = \mathbb R^m$, we recover the usual notion of Schwartz functions, i.e.\ rapidly decaying functions on $\mathbb R^m$ such that all partial derivatives of all orders are also rapidly decaying.

\begin{lem} \label{lem:schwartz general}
Let $M$ be an oriented real affine algebraic variety, and let $m = \dim (M)$.

\medskip \noindent (1) If $\Phi \in \Omega^m(M)$ is a Schwartz form of top degree, then the integral
\[ \int_{M} \Phi\]
exists (i.e.\ is  finite).

\medskip \noindent (2) If $\Phi \in \Omega^{m-1}(M)$ is a Schwartz form, then
\[
\int_{M} d \Phi = 0.
\]
The same conclusions hold if $M$ is replaced by any connected component $M'$ of $M$.
\begin{proof}  Both claims follow from straightforward calculus arguments when $M = \mathbb R^m$.

In general, suppose $M'$ is a connected component of $M$, and  $\Phi \in \Omega^{\bullet}(M')$.

By \cite[Remark I.5.12]{Shiota},  there exists a finite open cover $M' = U_1 \cup \cdots \cup U_j$ such that each $U_i$ is isomorphic to $\mathbb R^m$ as a Nash manifold. Applying a partition of unity, as in \cite[Theorem 4.4.1]{AG}, there are Schwartz forms $\Phi_i \in \Omega^{\bullet}(U_i)$ such that
\[
\Phi = \sum \Phi_i.
\]
In this way, we reduce both claims to the case $M = \mathbb R^m$.

\end{proof}
\end{lem}

We now apply this discussion to the case $M = B$, the group of upper triangular matrices, and the metrized bundle $E$, equipped with its connection $\nabla_E$.  To this end, we first endow $E$ with an algebraic structure.

Recall that $\rho \colon \GL_n(\mathbb R) \to O(\mfs_{n+1})$ denotes the representation \eqref{eqn:rho def}, and  that we had fixed a decomposition $\mathfrak s_{n+1} = V^+ \oplus V^-$ with $V^- := i\cdot  \Skew_{n+1}$. By definition, we have
\[\beta(b) = \rho(b)(V^-) \in  D\]
for any $b \in B$. By equivariance, $E$ is a trivial vector bundle with trivialization
\begin{equation}\label{eq:algebraic_structure_concrete}
\mr{triv}\colon B \times V^- \simlr E, \qquad (b, v^-) \longmapsto (b, \rho(b) v^-).
\end{equation}
Using this isomorphism, we identify $E$ with the $\mbR$-points of the trivial algebraic vector bundle $B\times V^-$. That is, a section $s\colon B\to E$ is algebraic if and only if the function $f\colon B\to V^-$ defining $\mr{triv}^{-1}\circ s$ is algebraic. In particular, all $B$-invariant sections of $E$ are algebraic, and any basis of the space of $B$-invariant sections provides an algebraic trivialization of $E$.

\begin{lem} \label{lem:alg forms}
(1) For any $y \in \mfs_{n+1}$, the sections $s_y$ and $\nabla_E( s_y)$ are algebraic.

\medskip \noindent (2) If $y$ is regular semi-simple, then $s_y$ is regular in the sense of Section \ref{section:MQ}.

\begin{proof}

It is a direct consequence of definitions that, in terms of \eqref{eq:algebraic_structure_concrete}, $s_y = \beta^* \wt s_y$ corresponds to the function
\[
f_y \colon B \lr V^-, \qquad f_y(b) = \mathrm{pr}_{V^-} \left(\rho (b)^{-1} y \right)
\]
which is evidently algebraic.

To show that $\nabla_E( s_y)$ is algebraic, we begin by fixing basis vectors $x_1, \dots, x_r$ for $V^-$, and extend them  $B$-invariant sections $\xi_1, \dots, \xi_r$ of $E$, so that $\xi_i(b) = \rho(b) x_i$.
We may therefore write
\[
s_y = \sum_{i=1}^r f_i \xi_i
\]
where each $f_i$ is an algebraic function on $B$. Thus
\[
\nabla_E(s_y) = \sum d f_i \otimes \xi_i  + f_i \, \nabla_E(\xi_i).
\]
Each section $\xi_i$ is algebraic, and since $\nabla_E$ is $B$-equivariant, each of the terms $\nabla_E(\xi_i)$ is $B$-invariant, and hence algebraic as well; we conclude that $\nabla_E(s_y)$ is algebraic, as required.

For a regular semisimple element $y$, the regularity of $s_y$ was already shown in Proposition \ref{prop:intersection_set} .
\end{proof}
\end{lem}

\begin{prop} \label{prop:psi and zeta are Schwartz}
(1) Suppose that $y \in \mfs_{n+1}$ is regular semi-simple. Then $\psi_y$ and $\zeta_y$ are Schwartz forms on $B$.

\medskip \noindent (2) Let $ρ \in C_c^{\infty}(B)$ be a function such that $ρ \equiv 1$ in a neighbourhood of $Z_s$, and let $f = 1-ρ$. Then $f g_s$ is a Schwartz form.
\begin{proof}
In the sequel, we use the abbreviations
$$s = s_y,\qquad q = n(n+1)/2,\quad \text{and}\quad C = (-1)^{q(q-1)/2} (2 \pi)^{-q/2}.$$
(1) By definition, we have
\[
\psi_y = C \{ e^{-2 \pi |s|^2 - 2 \sqrt \pi \nabla(s) - \kappa}  \} \qquad \text{and} \qquad 	ζ_y = C \{ s \wedge e^{-2 \pi |s|^2 - 2 \sqrt \pi \nabla(s) - \kappa}  \}.
\]
Since $|s|^2$ commutes with both $\nabla(s)$ and $\kappa$ in the algebra $\mathcal A$, we may rewrite these forms as
\[
\psi_y = C \, e^{-2 \pi |s|^2} \{ e^{- 2 \sqrt \pi \nabla(s) - \kappa}  \}, \qquad \zeta_y = C e^{-2 \pi |s|^2}\{ s \wedge e^{ - 2 \sqrt \pi \nabla(s) - \kappa} \}
\]
Note that the differential forms $ \{ e^{- 2 \sqrt \pi \nabla(s) - \kappa}  \}$ and $ \{ s \wedge e^{ - 2 \sqrt \pi \nabla(s) - \kappa} \}$ appearing above can be expressed as polynomial expressions in $s, \nabla(s)$ and $\kappa$, and hence are algebraic, by Lemma \ref{lem:alg forms}.
It therefore will suffice to show that the function
\[
\Phi(b) := e^{- 2 \pi | s(b)|^2  }
\]
is a Schwartz function on $B$. By definition of $s$ and the metric on $E$, this function is given by
$$Φ(b) = e^{4π Q(\mr{pr}_{V^-}(ρ(b)^{-1}(y))}.$$
We multiply it with the constant $e^{-2πQ(y)}$. Recall that $Q$ is invariant along the $G$-orbits of $\mfs_{n+1}$, implying $Q(y) = Q(ρ(b)^{-1}(y))$ for all $b\in B$. Moreover,  by the orthogonal decomposition $\mfs_{n+1} = V^+\oplus V^-$, we have
$$Q(v) = Q\big(\mr{pr}_{V^+}(v)\big) + Q\big(\mr{pr}_{V^-}(v)\big).$$
for any $v \in \mfs_{n+1}$.	Define the positive definite quadratic form $Q^* = Q\vert_{V^+} - Q\vert_{V^-}$. (In fact, this is the Siegel majorant.) Hence we see
\begin{equation} \label{eqn:norm relation}
|s(b)|^2 = Q^*(\rho(b)^{-1}  y) - Q(\rho(b)^{-1}  y) = Q^*(\rho(b)^{-1}  y) - Q(y)
\end{equation}
and therefore
$$\begin{aligned}
e^{-2πQ(y)} Φ(b) & =
e^{-2πQ^*(ρ(b)^{-1}(y))}.
\end{aligned}$$
In other words, we consider the Schwartz function $Φ'(v) = e^{-2πQ^*(v)}$ on $\mfs_{n+1}$, and pull it back to $B$ under the orbit map
$$B\lr \mfs_{n+1},\qquad b\longmapsto b\cdot y.$$
Since $y$ is regular semi-simple, the stabilizer of $y$ under the action of $B$ is trivial, and the orbit $B \cdot y$ is Zariski closed in $\mfs_{n+1}$. Moreover, the action $B \times \mfs_{n+1} \to \mfs_{n+1}$ is an algebraic map, and in particular, a closed immersion. Thus we obtain an identification $B \simeq B \cdot y$ as a closed Nash submanifold of $\mfs_{n+1}$.  As the restriction of a Schwartz function on a Nash manifold to a closed Nash submanifold is again a Schwartz function, we see that $\Phi = Φ'\vert_{B\cdot y}$ is Schwartz, concluding the proof.

\medskip \noindent (2)  Let $\rho\in C^{\infty}_c(B)$ be as in the statement, so that there is a neighbourhood $U$ of $Z_s$ with $\rho \equiv 1$ on $U$. Taking $f = 1-\rho$, so that in particular, $f \equiv 0$ on $U$. Using \eqref{eqn:gs homogeneous}, we may write
\[
f g_s =  \sum_{k=1}^{r} f   \left(  \int_{1}^{\infty} t^k e^{-2 \pi t^2|s|^2} \frac{dt}{t} \right) \eta_k(s)
\]
where each $\eta_k(s)$ is a polynomial expression in $s$, $\nabla(s)$ and $\kappa$. In particular, each $\eta_k(s)$ is algebraic.

Moreover, it follows from \eqref{eqn:norm relation} that there exists a constant $c> 0$ such that
\[
|s(b)|^2 > c \qquad \text{ for all } b \in B \setminus U.
\]
We then have that for any $k \geq 0$,
\begin{align*}
\left| f   \int_{1}^{\infty} t^k e^{-2 \pi t^2|s|^2} \frac{dt}{t} \right|  & < C   \int_{1}^{\infty} t^k e^{-2 \pi t^2|s|^2} \frac{dt}{t} \\
&< C   \int_{1}^{\infty} t^k e^{-2 \pi t^2| \left( \frac12 |s|^2 + \frac c2  \right) } \frac{dt}{t} \\
&< C e^{ -  \pi |s|^2} \int_{1}^{\infty} t^k e^{-\pi c t^2} \frac{dt}{t}
\end{align*}
The integral is finite, and the function $e^{- \pi |s|^2}$ is evidently bounded. Similar arguments show that
\[
D \left(  f   \int_{1}^{\infty} t^k e^{-2 \pi t^2|s|^2} \frac{dt}{t}  \right)
\]
is bounded for any algebraic differential operator $D$ on $B$, and the claim follows.
\end{proof}
\end{prop}

Let $B^0 \subset B$ denote the connected component of the identity, so that the map $B \to X$ restricts to an isomorphism
\begin{equation} \label{eqn:B0 X isom}
B^0 \simlr X.
\end{equation}
This isomorphism also identifies
\[
B^0 \cap Z_{s_y} \simlr X \cap D_y = \begin{cases} \mr{pt} & \text{if } y \text{ matches to signature } (n,0) \\ \emptyset & \text{otherwise} \end{cases}
\]
cf.\ Proposition \ref{prop:intersection_set}. Furthermore, in Section \ref{ss:orientations}, we had fixed orientations on $D$ and on $V^{-} = i \cdot \Skew_{n+1}$, which in turn determined an orientation on $D_y$, as well as on $X$ via Proposition \ref{prop:intersection_topological}. The orientation on $V^-$ determines an orientation on the bundle $\wt E$, and we have
\[
[Z_{s_y}] = [D_y],
\]
where $[Z_{s_y}]$ is oriented according to Definition \ref{def:orientation_on_Z_s}.
On the other hand,  using \eqref{eqn:B0 X isom} to transfer the orientation on $X$ to $B^0$, we have
\begin{equation} \label{eqn:deg Z_s}
\deg \left([Z_{s_y}] \right)= [X] \cdot_{[D]} [D_y] =  \begin{cases}ε(y) & \text{if $y$ matches to signature } (n,0) \\ 0 & \text{otherwise}. \end{cases}
\end{equation}

\begin{thm} \label{thm:integral_of_psi_y}
Let $y \in \mfs_{n+1}$ be a regular semi-simple element. Then
\[
\int_{B^0} \psi_y = \begin{cases}ε(y) & \text{if $y$ matches to signature } (n,0) \\ 0 & \text{otherwise}. \end{cases}
\]
\begin{proof} If $B^0$ were compact, this identity would follow immediately from Proposition \ref{prop: Green current} by evaluating \eqref{eq:current_equation} at the constant function $1$. In our case, however, we will need to be a bit more indirect.

Let us again abbreviate $s = s_y$. We begin by fixing a sequence of successively relatively compact open neighbourhoods
\[
U_1 \subset U_2 \subset \cdots
\]
of $Z_s$ such that $\cup\, U_k = B^0$. (If $Z_s$ is empty, we fix an arbitrary family of nested relatively compact open sets exhausting $B^0$.) We choose a family of compactly supported functions $ρ_k \in C_c^{\infty}(B_0)$ such that
\[
|ρ_k(b) | \leq 1 \qquad \text{ for all } b \in B^0
\]
and
\[
ρ_k \equiv 1 \qquad \text{ on } U_k.
\]
Then, by Proposition \ref{prop: Green current}, we have
\begin{align*}
\int_{B^0} \psi_s &= \lim_{k \to \infty} \int_{B^0} \psi_s ρ_k\\
&= \lim_{k \to \infty}  \delta_{[Z(s)]} (ρ_k) + 2 \sqrt \pi \lim_{k \to \infty} \int_{B^0} g_s \wedge d ρ_k.
\end{align*}
By construction, we also have for all $k$ that
\[
\delta_{[Z(s)]}(ρ_k) = \deg([Z_s]) = \begin{cases} ε(y) & \text{if $y$ matches to signature } (n,0) \\ 0 & \text{otherwise}. \end{cases}
\]
On the other hand, let $f\in C^\infty(B^0)$ be such that $f\equiv 0$ on $U_1$ and $f\equiv 1$ on $B^0\setminus U_2$. If $k \geq 3$, then $d ρ_k$ is supported on $B^0 \setminus U_2$, so
\[
\int_{B^0}g_s \wedge d ρ_k = \int_{B^0} (f g_s) \wedge d ρ_k = \int_{B^0} d (f g_s) \wedge ρ_k.
\]
By Proposition  \ref{prop:psi and zeta are Schwartz} and Lemma \ref{lem:schwartz general}, we have
\[
\lim_{k \to \infty} \int_{B^0} d(f g_s) \wedge ρ_k = \int_{B^0} d(f g_s) = 0,
\]
concluding the proof of the theorem.
\end{proof}
\end{thm}

\subsection{Differential forms and orbital integrals} \label{sec:main thm}
As a final step, we interpret Theorem \ref{thm:integral_of_psi_y} in terms of orbital integrals. Consider the Iwasawa decomposition $G = B^0 \cdot O(n)$. We obtain a decomposition $dg = db \, dk$ of measures, where $dg$ is our fixed Haar measure on $G$, where $dk$ is the Haar measure on $O(n)$ normalized to have total volume 1, and where $db$ is the left Haar measure determined by the integral formula
\[
\int_{G} f(g) dg \ = \ \int_{B^0} \int_{O(n)} f(bk) \, dk \, db
\]
for any integrable function $f$.

Next, recall that we have fixed an orientation on $B^0 \simeq X$ as in Section \ref{ss:orientations}. Let $\wt \omega \in \Omega^{top}(B)$ be a left-invariant (algebraic, in particular) differential form of top degree whose restriction to $B^0$ is positive, and induces the measure $db$; i.e.\ for any integrable function $f$ on $B$, we have
\[
\int_{B^0} f \wt \omega = \int_{B^0} f(b) db.
\]
Let $\mfb = \Lie(B)$, and let $ω\in \det(\mfb)$ be the value $ω = \wt {ω}_e$ of $\wt{ω}$ at the identity element. Consider the pullback map induced by $β:B\to SO(\mfs_{n+1})$,
$$β^*:\mcS(\mfs_{n+1}) \tensor \wedge^{n(n+1)/2} (\mfP^*) \lr \mcS(\mfs_{n+1})\tensor \det(\mfb),$$
and define $Φ\in \mcS(\mfs_{n+1})$ by the identity
$$2^{n(n+1)/4}\cdot β^*(φ_\KM) = Φ\tensor ω.$$

\begin{lem} \label{lem:O(n) equiv of PhiKM}
For $k \in O(n)$ and $y \in \mfs_{n+1}$, we have $\Phi(k^{-1} \cdot y) = \eta(k) \Phi(y)$.
\begin{proof}
Let $V^+ = i\cdot \Sym_{n+1}$ and $V^- = i\cdot \Sym_{n+1}$ again denote our usual choice of maximal definite subspaces of $\mfs_{n+1}$. The composition
$$O(n) \overset{α}{\lr} O(V^+)\times O(V^-) \overset{ν}{\lr} \{\pm 1\}$$
equals $η^n$, because the element $σ = \diag(-1,1,\ldots,1)$ acts with determinant $(-1)^n$ on $V^-$. Moreover, the identification $B_0\simto X$ induces an isomorphism $\mfb\simto T_eX = \Sym_n$, and hence via composition an action of $O(n)$ on $\mfb$. Its determinant, which is simply the determinant of $O(n)$ acting on $\Sym_n$, is $η^{n-1}$.

From the invariance property in \eqref{eqn:PhiKM equivariance}, we obtain for general $k\in O(n)$ that
$$\begin{aligned}
Φ(k^{-1}v)\tensor ω & = η(k)^n\cdot Φ(v)\tensor (k^*ω)\\
& = η(k) Φ(v) \tensor \omega
\end{aligned}$$

\end{proof}
\end{lem}

We now state our main theorem:
\begin{thm} \label{thm:main body}
Suppose $y \in \mfs_{n+1}$ is regular semi-simple.
Then
\[
\Orb(y, \Phi) =  \begin{cases} e^{ - 2\pi Q(y)} & \text{if $y$ matches to signature } (n,0) \\ 0 & \text{otherwise.} \end{cases}
\]
In other words, $\Phi$ is a Gaussian test function.

\begin{proof}
By definition, we have
\begin{align*}
\Orb(y, \Phi) & = ε(y) \int_{G} \Phi(g^{-1} \cdot y) \, \eta(g) \, dg \\
&= ε(y) \int_{B^0} \int_{O(n)}  \Phi(k^{-1} b^{-1 }\cdot y) \, \eta(k) \, dk db \\
&= ε(y) \int_{B^0} \Phi(b^{-1}  \cdot y) db
\end{align*}
where in the last line, we use the fact that $\Phi(k\cdot y) \eta(k) = \Phi(y)$, as in Lemma \ref{lem:O(n) equiv of PhiKM}.

Applying our conventions on measures, as well as Theorem \ref{thm:Brancherau}, we have
\begin{align*}
\int_{B^0} \Phi(b^{-1} \cdot y) db &= \int_{B^0} \Phi(b^{-1} \cdot y) \, \wt \omega \\
&= \int_{B^0} 2^{n(n+1)/4} \beta^* (\wt \varphi_{\KM}(y)) \\
&= e^{-2 \pi Q(y)}  \int_{B^0} \psi_y.
\end{align*}
The result now follows from Theorem \ref{thm:integral_of_psi_y}.
\end{proof}
\end{thm}

\begin{ex}[Case $n = 1$]\label{ex:1}
In this situation our construction recovers the test function used in \cite[\S12]{Zhang_AFL}. Consider $G = \mbR^\times$ with Haar measure $dt/t$. We denote the coordinates on $\mfs_2$ by
$$\mfs_2 = i\cdot \left.\left\{\begin{pmatrix} a & y_1 \\ y_2 & d \end{pmatrix} \right\vert a, d, y_1, y_2\in \mbR\right\}.$$
Then $a$ and $d$ are $G$-invariant. We pick a transfer factor $ε$ that is positive whenever $y_1$ and $y_2$ are positive. The quadratic form $Q(y) = -\mr{tr}(y^2)$ is give by
$$Q(y) = a^2 + d^2 + 2y_1y_2.$$
Its definite components are
$$Q_+(y) = a^2 + d^2 + \frac{(y_1 + y_2)^2}{2},\quad Q_- = \frac{(y_1 - y_2)^2}{2}.$$
The Siegel--Gaussian is hence
$$e^{-2π(Q_+ + Q_-)} = e^{-2π(a^2 + d^2 + y_1^2 + y_2^2)}.$$
The Schwartz function $Φ$ from Theorem \ref{thm:main body} is
$$Φ(y) = 2^{-1/2}\cdot (y_1 + y_2) e^{-2π(a^2 + d^2 + y_1^2 + y_2^2)}.$$
The theorem now states that
$$\Orb\left(i\cdot \begin{pmatrix} a & y_1 \\ y_2 & d\end{pmatrix}, Φ\right) = e^{-2π(a^2 + d^2)} \cdot \begin{cases} e^{-4πy_1y_2} & \text{if $y_1y_2 > 0$}\\
0 & \text{otherwise.}\end{cases}$$
\end{ex}

\begin{ex}[Case $n = 2$] \label{ex:2}
We first normalize the Haar measure on $G = \GL_2(\mbR)$. Use the following notation for the coordinates of the Iwasawa decomposition of an element $g\in G$:
$$g = \begin{pmatrix}
a_1 & \\ & a_2
\end{pmatrix}\cdot \begin{pmatrix}
1 & b \\ & 1
\end{pmatrix}\cdot θ,\quad a_1, a_2 \in \mbR_{>0},\ b\in \mbR,\ θ \in O(2).$$
Then fix the measure as
$$dg = \frac{da_1\,da_2\,db\,dθ}{a_1a_2},\quad \int_{O(2)} dθ = 1.$$
We denote the coordinates on $\mfs_3$ by
$$\mfs_3 = i\cdot \left.\left\{\begin{pmatrix} y_{11} & y_{12} & v_1 \\ y_{21} & y_{22} & v_2 \\ w_1 & w_2 & d \end{pmatrix} \right\vert \text{all entries in $\mbR$}\right\}.$$
The entry $d$ is $G$-invariant. We pick a transfer factor $ε$ that is positive on elements of the form \eqref{eq:y_diagonal} with $λ_1 > λ_2$, $μ_1 > 0$, and $μ_2 >0$. The quadratic form $Q(y) = -\mr{tr}(y^2)$ is given by
$$Q(y) = y_{11}^2 + y_{22}^2 + d^2 + 2(y_{12}y_{21} + v_1w_1 + v_2w_2).$$
Its definite components are
$$\begin{aligned}
Q_+(y) & = y_{11}^2 + y_{22}^2 + d^2 + \frac{(y_{12} + y_{21})^2 + (v_1 + w_1)^2 + (v_2 + w_2)^2}{2}\\[1mm]
Q_-(y) & = \frac{(y_{12} - y_{21})^2 + (v_1 - w_1)^2 + (v_2 -w_2)^2}{2}.
\end{aligned}$$
The Siegel--Gaussian is hence
$$e^{-2π(Q_+ + Q_-)} = e^{-2π(y_{11}^2 + \ \ldots\ +  d^2)} =e^{-2 \pi \, \tr ({}^ty\cdot y)}$$
where the exponent involves the sum of the squares of all $9$ entries of $y$. The Schwartz function $Φ$ from Theorem \ref{thm:main body} is
$$\begin{aligned}
Φ(y) = & \left[ \sqrt{2} \cdot(y_{11} - y_{22})(v_1 + w_1)(v_2 + w_2)\phantom{\frac{1}{\sqrt{2}}}\right.\\
& \left. - \frac{1}{\sqrt{2}} (y_{12} + y_{21})\left((v_1 + w_1)^2 - (v_2 + w_2)^2\right)\right] \cdot e^{-2π(Q_+ + Q_-)}.
\end{aligned}$$
\end{ex}

\section{From Lie algebras to Lie groups}
\label{s:group}

Recall the definition of $S_{n+1}$ from \eqref{eq:def_S_n_plus_1}. Given a Schwartz function $ϕ\in \mcS(S_{n+1})$ and a regular semi-simple element $γ \in S_{n+1}$, there is the Jacquet--Rallis orbital integral
\begin{equation}\label{eq:def_orb_int_group}
\Orb(γ, ϕ) = \epsilon(γ) \int_{\GL_n(\mbR)} ϕ(g^{-1}γg) η(g)\,dg,
\end{equation}
where $\epsilon(γ)\in \mbC^\times$ will be defined in \S\ref{ss:epsilon}. Our next aim is to prove the following theorem.

\begin{thm}\label{thm:main_group}
There exists a Gaussian Schwartz function on $S_{n+1}$. That is, there exists a Schwartz function $ϕ \in \mcS(S_{n+1})$ such that for all regular semi-simple $γ\in S_{n+1}$
\begin{equation}\label{eq:Gaussian_group}
\Orb(γ, ϕ) = \begin{cases}
1 & \text{if $γ$ matches to signature $(n,0)$}\\
0 & \text{otherwise.}
\end{cases}
\end{equation}
\end{thm}

The proof will be given in \S\ref{ss:proof_group}.

\subsection{Transfer factors}
\label{ss:epsilon}

We completely follow the conventions in the literature, see \cite{Zhang_GGP} and \cite{Xue}. Fix a character $η':\mbC^\times \to \mbC^\times$ that extends $η:\mbR^\times \to \{\pm 1\}$. Note that all such characters are of the form $z\mapsto (z/|z|)^m$ with $m\in 1+2\mbZ$. Let us write $S_{n+1,\rs}\subseteq S_{n+1}$ for the subset of regular semi-simple elements. Let $e$ be be the row vector $(0,\ldots,0,1) \in \mbC_{n+1}$. Define the transfer factor on $S_{n+1}$ by
\begin{equation}\label{eq:def_transfer_factor_group}
\begin{aligned}
\epsilon: S_{n+1,\rs} & \lr \mbC^\times\\
γ & \longmapsto η'\big(\det(γ)^{-\lfloor (n+1)/2 \rfloor} \det\big((e, eγ,\ldots, eγ^n)\big)\big).
\end{aligned}
\end{equation}
Here, $(e, eγ, \ldots, eγ^n)$ denotes the matrix with $eγ^i$ in the $(i+1)$-th row. Note that unlike in the Lie algebra setting, $\epsilon$ need not be locally constant. It still satisfies $\epsilon(hγh^{-1}) = η(h)\epsilon(γ)$ for all $h\in \GL_n(\mbR)$ and $γ\in S_{n+1,\rs}$. The definition of \eqref{eq:def_orb_int_group} is now complete.

We would like to compare $\epsilon$ with a transfer factor on the Lie algebra. So, again following \cite{Zhang_GGP} and \cite{Xue}, we make the following explicit choice:
\begin{equation}\label{eq:def_transfer_factor_Lie}
\begin{aligned}
ε: \mfs_{n+1,\rs} & \lr \{\pm 1\}\\
y & \longmapsto η\big((-i)^{n(n+1)/2}\cdot \det\big((e, ey,\ldots, ey^n)\big)\big).
\end{aligned}
\end{equation}

\begin{defn}\label{def:Cayley}
Let $ξ \in \mbC^1$ be an element of norm $1$. Define open subsets $S_{n+1,ξ}$ and $\mfs_{n+1,ξ}$ by the condition
\begin{equation}\label{eq:def_det_invertible}
\det(γ-ξ) \neq 0\quad \text{resp.}\quad \det(y-ξ) \neq 0.
\end{equation}
We write $S_{n+1,\rs,ξ}$ and $\mfs_{n+1,\rs,ξ}$ for their subsets of regular semi-simple elements. The Cayley transform (with parameter $ξ$) is the isomorphism
$$\begin{aligned}
c_ξ\colon S_{n+1,ξ} & \simlr \mfs_{n+1,1}\\
γ & \longmapsto \frac{γ + ξ}{γ - ξ}.
\end{aligned}$$
Its inverse is given by $y \mapsto ξ(y + 1)/(y - 1)$. Note that $S_{n+1,ξ}$ and $\mfs_{n+1,ξ}$ are $\GL_n(\mbR)$-stable, and that $c_ξ$ is $\GL_n(\mbR)$-equivariant. In particular, it preserves the property of being regular semi-simple, inducing an isomorphism
$$c_ξ:S_{n+1,\rs,ξ}\simlr \mfs_{n+1,\rs,1}.$$
\end{defn}

It is shown in \cite[Lemma 3.5]{Zhang_GGP} that $\epsilon$ and $ε$ are compatible under the Cayley transform. There is a small typo in the argument when $n$ is odd, so we include a full statement.

\begin{lem}[\protect{\cite[Lemma 3.5]{Zhang_GGP}}]\label{lem:compatibility_transfer_factors}
Let $ξ\in \mbC^1$ be an element of norm $1$. Then $\epsilon$ on $S_{n+1,\rs,ξ}$ and $ε$ on $\mfs_{n+1,\rs,1}$ are compatible under $c_ξ$ in the sense that there exists a smooth, algebraic, nowhere vanishing function $ρ_ξ$ on $S_{n+1,ξ}$ such that for all regular semi-simple $γ\in S_{n+1,\rs,ξ}$,
$$ε(c_ξ(γ)) = ρ_\xi(γ)\cdot \epsilon(γ).$$
\end{lem}
\begin{proof}
Set $T = 2ξ/(γ-ξ)$ and observe that $(γ+ξ)/(γ-ξ) = 1 + T$. Substituting in definitions, we have
$$ε(c_ξ(γ)) = η'(-i)^{n(n+1)/2} \cdot η'\big(\det((e,e(1+T),\ldots e(1+T)^n))\big).$$
By elementary row operations, we find
$$\det\big((e,e(1+T),\ldots, e(1+T)^n\big)) = \det\big((e,eT,\ldots,eT^n)\big).$$
Multiplying with $(γ-ξ)^n$ from the right, and writing $r(ξ, y) = (2ξ)^{n(n+1)/2} \det(γ-ξ)^{-n}$, we have
$$\det\big((e,eT,\ldots,eT^n)\big) = r(ξ, γ)\det\big((e(γ-ξ)^n, e(γ-ξ)^{n-1},\ldots,e(γ-ξ),e)\big).$$
Again by elementary row operations, this last determinant equals
$$\det\big((eγ^n,eγ^{n-1},\ldots,eγ,e)\big) = (-1)^{n(n+1)/2}\det\big((e,eγ,\ldots,eγ^n)\big).$$
Applying $η'$ to this expression yields
$$(-1)^{n(n+1)/2}η'(\det(γ)^{\lfloor (n+1)/2\rfloor}) \epsilon(γ).$$
Hence, overall, we obtain $ε(c_ξ(γ)) = ρ_ξ(γ)\cdot \epsilon(γ)$ with
\begin{equation}
\begin{aligned}
ρ_ξ(γ) & = η'\big[(-i)^{n(n+1)/2}\cdot (2ξ)^{n(n+1)/2}\cdot \det(γ-ξ)^{-n}\cdot (-1)^{n(n+1)/2}\cdot \det(γ)^{\lfloor (n+1)/2\rfloor}\big]\\[1mm]
& = η'\big[(2iξ)^{n(n+1)/2} \cdot \det(γ-ξ)^{-n}\cdot \det(γ)^{\lfloor (n+1)/2\rfloor} \big].
\end{aligned}
\end{equation}\label{eq:rho_final}
\end{proof}
\begin{rmk}
For any $ξ\in \mbC^1$, multiplication by $ξ$ induces an isomorphism $S_{n+1,1}\simto S_{n+1,ξ}$, and $c_ξ(ξγ) = c_1(γ)$. Assume that $n$ is even, which implies $\lfloor (n+1)/2\rfloor = n/2$. The normalization of $\epsilon$ is then such that $\epsilon(ξγ) = \epsilon(γ)$. Thus $ρ_ξ(ξγ) = ρ_1(γ)$. Specializing \eqref{eq:rho_final} to $ξ = 1$, we have
$$ρ_1(γ) = η'\big[(2i)^{n(n+1)/2} \cdot \det(γ-1)^{-n}\cdot \det(γ)^{n/2}\big].$$
Recall that $\overbar{γ} = γ^{-1}$ by definition of $S_{n+1}$, so we find that
$$\begin{aligned}
\overline{\det(γ-1)^{-n}\cdot \det(γ)^{n/2}} & = \frac{γ^{-n/2}}{(γ^{-1}-1^{-1})^n}\\[1mm]
& = \det(γ-1)^{-n}\cdot \det(γ)^{n/2}.
\end{aligned}$$
Hence, this quantity lies in $\mbR^\times$. This shows that $ρ_ξ(γ)$ is a locally constant function when $n$ is even, which also implies that $\epsilon$ is locally constant. In fact, since $ρ_ξ(γ)$ then comes from an algebraic function on a connected variety, one even obtains that $ρ_ξ(γ)$ is constant as stated in \cite[(3.4.8)]{BP_GGP}.
\end{rmk}

\subsection{Proof of Theorem \ref{thm:main_group}}
\label{ss:proof_group}

We now prove Theorem \ref{thm:main_group} for which we first introduce some terminology related to categorical quotients. References for the following in the Lie algebra case are \cite[\S3.1]{Zhang_GGP} or \cite[\S2]{Chaudouard}, while the group case is discussed in \cite[\S2.1]{Zhang_AFL}.

Let $Q^{\mr{alg}}_{n+1}$ and $\mfq^{\mr{alg}}_{n+1}$ denote the categorical quotients by $\GL_n$ of the algebraic varieties underlying $S_{n+1}$ and $\mfs_{n+1}$. For example, the invariant map \eqref{eq:def_invariant_s}, which is clearly algebraic, identifies $\mfq^{\mr{alg}}_{n+1}$ with $\mbA^{2n+1}$. A similar description exists for $Q^{\mr{alg}}_{n+1}$, but we will not need it.

For every $n$-dimensional hermitian $\mbC$-vector space $V$, the quotients $Q^{\mr{alg}}_{n+1}$ and $\mfq^{\mr{alg}}_{n+1}$ can also be identified with the categorical quotients by the algebraic group defining $U(V)$ of the algebraic varieties underlying $U(V\oplus \mbC)$ and $\mfu(V\oplus \mbC)$. These identifications can be chosen so that they realize the matching bijections in \eqref{eq:intro_matching_group}.

Concretely, let $Q_{n+1} = Q^{\mr{alg}}(\mbR)$ and $\mfq_{n+1} = \mfq^{\mr{alg}}(\mbR)$ denote the $\mbR$-points. After the mentioned identifications, the quotient maps
$$\begin{array}{rrl}
\mr{inv}: & S_{n+1},\ U(V\oplus \mbC) & \!\lr Q_{n+1},\\[1mm]
\mr{inv}: & \mfs_{n+1},\ \mfu(V\oplus \mbC) &\! \lr \mfq_{n+1}
\end{array}$$
realize the matching bijections in the sense that
$$\begin{array}{rcccl}
\GL_n(\mbR) \back\!\back S_{n+1,\rs} & \underset{\mr{inv}}{\simlr} & Q_{n+1,\rs} & \underset{\mr{inv}}{\overset{\sim}{\longleftarrow}} & \coprod_{r + s = n} U(V_{(r,s)}) \back\!\back U(V_{(r,s)} \oplus \mbC)_\rs,\\[3mm]
\GL_n(\mbR) \back\!\back \mfs_{n+1,\rs} & \underset{\mr{inv}}{\simlr} & 
\mfq_{n+1,\rs} & \underset{\mr{inv}}{\overset{\sim}{\longleftarrow}} & \coprod_{r + s = n} U(V_{(r,s)}) \back\!\back \mfu(V_{(r,s)} \oplus \mbC)_\rs.
\end{array}
$$

Recall that our aim is to construct a Schwartz function $ϕ \in \mcS(S_{n+1})$ such that \eqref{eq:Gaussian_group} holds. This construction can be performed locally in the following sense: Let $T\subseteq Q_{n+1}$ be the image $\mr{inv}(U(V_{(n,0)}\oplus \mbC))$, which is compact. Assume that $T\subseteq \cup_{i = 1}^r W_i$ is an open covering of $T$ in $Q_{n+1}$ and that $ϕ_i\in \mcS(S_{n+1})$ are such that for all regular semi-simple $γ\in \mr{inv}^{-1}(W_i)$,
$$\Orb(γ, ϕ_i) = \begin{cases} 1 & \text{if $γ$ matches to signature $(n,0)$}\\
0 & \text{otherwise.}\end{cases}$$
Let $λ_i\in C_c^\infty(W_i)$ be such that $(\sum_{i = 1}^r λ_i)\vert_T \equiv 1$. Then $\sum_{i = 1}^r \mr{inv}^*(λ_i)\cdot ϕ_i$ satisfies \eqref{eq:Gaussian_group} and the proof of Theorem \ref{thm:main_group} is complete. Our task is hence to construct the datum $(W_i, ϕ_i)_{i = 1}^r$. Since $T$ is compact, it suffices for each $t\in T$ to construct an open neighborhood $W\subseteq Q_{n+1}$ and a Schwartz function $ϕ_W$ satisfying \eqref{eq:Gaussian_group} for regular semi-simple $γ\in \mr{inv}^{-1}(W)$.

The definition of $S_{n+1,ξ}$ and $\mfs_{n+1,ξ}$ by \eqref{eq:def_det_invertible} is in terms of the $G$-invariant polynomials $\det(γ-ξ)$ and $\det(y-ξ)$. There are hence Zariski open subsets $Q_{n+1,ξ}$ and $\mfq_{n+1,ξ}$ of $Q_{n+1}$ (resp. $\mfq_{n+1}$) such that
$$S_{n+1,ξ} = \mr{inv}^{-1}(Q_{n+1,ξ}),\quad \mfs_{n+1,ξ} = \mr{inv}^{-1}(\mfq_{n+1,ξ}).$$
Choose $ξ \in \mbC^1$ with $t\in Q_{n+1,ξ}$ and consider the Cayley transform
$$c_ξ:S_{n+1,ξ}\simlr \mfs_{n+1,1}.$$
By $G$-equivariance of $c_ξ$, for any Schwartz function $Φ\in \mcS(\mfs_{n+1})$ and regular semi-simple $γ\in S_{n+1,\rs,ξ}$, we have
$$\epsilon(γ)^{-1} \Orb(γ, c_ξ^*(Φ)) = ε(c_ξ(γ))^{-1}\Orb(c_ξ(γ), Φ).$$
Let $λ\in C^\infty_c(Q_{n+1,ξ})$ be any compactly supported function. Then $\mr{inv}^*(λ)\cdot c_ξ^*(Φ)$ is a Schwartz function on $S_{n+1}$ with support in $S_{n+1,ξ}$. The ratio
$$ε(c_ξ(γ)) / \epsilon(γ),\qquad γ\in S_{n+1,ξ}$$
is defined on all of $S_{n+1,ξ}$ by Lemma \ref{lem:compatibility_transfer_factors}, and hence
$$ϕ_λ := \frac{ε(c_ξ(γ))}{\epsilon(γ)} \cdot \mr{inv}^*(λ)\cdot c_ξ^*(Φ)$$
lies in $\mcS(S_{n+1})$ and satisfies
\begin{equation}\label{eq:key_technical}
\Orb(γ,ϕ_λ) = λ(\mr{inv}(γ))\cdot\Orb(c_ξ(γ), Φ),\qquad γ\in S_{n+1,\rs,ξ}.
\end{equation}

We now make the following choices. Let $Φ \in \mcS(\mfs_{n+1})$ be the Gaussian test function from Theorem \ref{thm:main body}. Choose $λ$ such that $λ\equiv c_ξ^*(e^{2πQ})$ on a neighborhood $W$ of $t$. (Here, the quadratic form $Q(y) = -\mr{tr}(y^2)$ has been viewed as a function on $\mfq_{n+1}$.) Then \eqref{eq:key_technical} specializes to
$$\begin{aligned}
\Orb(γ,ϕ_W) & = e^{2πQ(c_ξ(γ))} \Orb(c_ξ(γ), Φ)\\
& = \begin{cases} 1 & \text{if $γ\in \mr{inv}^{-1}(W)$ matches to signature $(n,0)$}\\
0 & \text{otherwise.}
\end{cases}
\end{aligned}$$
Here, in the final step, we have used the following lemma.
\begin{lem}
Let $γ \in S_{n+1,\rs,ξ}$ be a regular semi-simple element. Then $γ$ matches to signature $(r,s)$ if and only if $c_ξ(γ)$ matches to signature $(r,s)$.
\end{lem}
\begin{proof}
The Cayley transform can also be defined on the unitary side by the same formulas. The statement then follows from the definition of matching.
\end{proof}

This completes the proof of Theorem \ref{thm:main_group}. \qed

\section{Transfer of polynomial type Schwartz functions}

Finally, we extend Theorems \ref{thm:main body} and \ref{thm:main_group} from Gaussian test functions to polynomial type test functions. We write $U(n)$, $\mfu(n)$, etc. instead of $U(V_{(n,0)})$, $\mfu(V_{(n,0)})$, etc. in the following.

Normalize the Haar measure on $U(n)$ to have total volume $1$. For a Schwartz function $ψ$ on $\mfu(n+1)$ and a regular semi-simple element $x\in \mfu(n+1)_{\mr{rs}}$, the orbital integral is defined by
$$\Orb(x, ψ) = \int_{U(n)} ψ(g^{-1}xg)\,dg.$$
The exact same definition applies to a smooth function $ψ$ on $U(n+1)$ and a regular semi-simple element $δ\in U(n+1)_\rs$. This defines the right hand sides in \eqref{eq:polynomial_transfer_Lie} and \eqref{eq:polynomial_transfer_group}.

Let $\mbU(n)$ and $\mbU(n+1)$ be the algebraic groups defining $U(n)$ and $U(n+1)$. Since $U(n+1)$ and $\mbU(n+1)$ are both connected, the restriction map from algebraic functions to smooth functions is injective. Hence, we can view $\mbR[\mbU(n+1)]$ as a subspace of $C^\infty(U(n+1))$. A smooth function on $U(n+1)$ is said to be \emph{algebraic} if it lies in this subspace. Similar terminology applies to $\mfu(n+1)$, $Q_{n+1}$ and $\mfq_{n+1}$.

\begin{lem}\label{lem:U_inv_polynomials}
Let $p$ be an $U(n)$-invariant algebraic function on $U(n+1)$ or $\mfu(n+1)$. Then there exists an algebraic function $f$ on $Q_{n+1}$ resp. $\mfq_{n+1}$ such that $p = \mr{inv}^*(f)$.
\end{lem}
\begin{proof}
The function $p$ being $U(n)$-invariant and algebraic means by definition that it lies in the ring $\mbR[U(n+1)]^{U(n)}$ or $\mbR[\mfu(n+1)]^{U(n)}$, respectively. Since $U(n)$ and $\mbU(n)$ are both connected, these two rings agree with the algebraic invariants $\mbR[U(n+1)]^{\mbU(n)}$ resp. $\mbR[\mfu(n+1)]^{\mbU(n)}$. The categorical quotients by $\mbU(n)$ of $\mbU(n+1)$ and (the algebraic variety underlying) $\mfu(n+1)$ are, by definition, the spectra
\begin{equation}\label{eq:def_cat_quot}
Q^{\mr{alg}}_{n+1} = \Spec \mbR[\mbU(n+1)]^{\mbU(n)},\qquad \mfq^{\mr{alg}}_{n+1} = \Spec \mbR[\mfu(n+1)]^{\mbU(n)}.
\end{equation}
It is then tautological that $p$ comes by pullback along the quotient map as claimed.
\end{proof}

Theorems \ref{thm:intro_extension_polynomial} and \ref{thm:intro_main_group} are now proved in the exact same way with Lemma \ref{lem:U_inv_polynomials}. We only present the details in the Lie group case for brevity:

\begin{thm}\label{thm:matrix_coeff}
Let $ψ$ be an algebraic function on $U(n+1)$. Then there exists a Schwartz function $ϕ\in \mcS(S_{n+1})$ such that for all $γ\in S_{n+1,\rs}$,
$$\Orb(γ, ϕ) = \begin{cases}
\Orb(δ, ψ) & \text{if $γ$ matches an element $δ\in U(n+1)$}\\[1mm]
0 & \text{otherwise.}
\end{cases}$$
\end{thm}
\begin{proof}
By assumption, $ψ$ lies in the ring $\mbR[\mbU(n+1)]$ of regular functions on $\mbU(n+1)$. The averaged function $\overline{ψ}(δ) := \int_{U(n)} ψ(g^{-1}δg)$ then lies in the invariants $\mbR[\mbU(n+1)]^{U(n)}$. Moreover, $ψ$ and $\ov{ψ}$ have the same orbital integrals. Applying Lemma \ref{lem:U_inv_polynomials}, there exists an algebraic function $f$ on $Q_{n+1}$ such that $\overline{ψ} = \mr{inv}^*(f)$, and we find
$$\Orb(δ, ψ) = f(\mr{inv}(δ))$$
for all $δ\in U(n+1)_\rs$.

Let $Φ$ be a Gaussian test function on $S_{n+1}$ as in Theorem \ref{thm:main_group}. Then, for every $γ\in S_{n+1,\rs}$, we find
$$\begin{aligned}
\Orb(γ, \mr{inv}^*(f)\cdot Φ) & \ =\ f(\mr{inv}(γ)) \Orb(γ, Φ)\\[1mm]
&\ =\ \begin{cases} \Orb(δ, ψ) & \text{if $γ$ matches an element $δ\in U(n+1)$}\\
0 & \text{otherwise.}
\end{cases}\end{aligned}$$
In other words, $ϕ = \mr{inv}^*(f)\cdot Φ$ satisfies the requirements of the theorem.
\end{proof}

\end{document}